\documentclass[a4paper,11pt]{amsart}

\newfont{\cyr}{wncyr10 scaled 1100}

\usepackage[usenames,dvipsnames]{color}

\usepackage[left=2.7cm,right=2.7cm,top=3.5cm,bottom=3cm]{geometry}
\usepackage{amsthm,amssymb,amsmath,amsfonts,mathrsfs,amscd,yfonts}
\usepackage{turnstile}
\usepackage[latin1]{inputenc}
\usepackage[all]{xy}
\usepackage{latexsym}
\usepackage{longtable}
\usepackage{hyperref}

\theoremstyle{plain}
\newtheorem{theorem}{Theorem}[section]

\newtheorem{propo}[theorem]{Proposition}
\newtheorem{coro}[theorem]{Corollary}
\newtheorem{conj}[theorem]{Conjecture}

\theoremstyle{definition}

\newtheorem{question}[theorem]{Question}
\newtheorem{examplewr}[theorem]{Example}

\theoremstyle{remark}
\newtheorem{obswr}[theorem]{Observation}
\newtheorem{remarkwr}[theorem]{Remark}

\theoremstyle{plain}
\newtheorem*{pro*}{Proposition}
\newtheorem*{rem*}{Remark}

\newenvironment{remark}{\begin{remarkwr}\begin{upshape}}{\end{upshape}\end{remarkwr}}

\DeclareMathOperator{\wt}{wt}

\DeclareMathOperator{\BK}{BK}

\DeclareMathOperator{\id}{Id}

\DeclareMathOperator{\pr}{pr}
\DeclareMathOperator{\cyc}{cyc}
\DeclareMathOperator{\fin}{f}

\DeclareMathOperator{\nr}{nr}

\DeclareMathOperator{\Spf}{Spf}

\DeclareMathOperator{\Reg}{Reg}

\DeclareMathOperator{\lcm}{lcm}
\DeclareMathOperator{\bal}{bal}

\DeclareMathOperator{\loc}{loc}

\DeclareMathOperator{\ad}{ad}
\DeclareMathOperator{\PR}{PR}

\DeclareMathOperator{\Sch}{Sch}

\newcommand{\Lp}{{\mathscr{L}_p}}

\newcommand{\cW}{\mathcal W}

\newcommand{\Q}{\mathbb{Q}}
\newcommand{\Z}{\mathbb{Z}}

\newcommand{\Gal}{\mathrm{Gal\,}}
\newcommand{\GL}{\mathrm{GL}}

\newcommand{\Frob}{\mathrm{Fr}}

\newcommand{\Fr}{\mathrm{Fr}}

\newcommand{\ord}{{\mathrm{ord}}}
\newfont{\gotip}{eufb10 at 12pt}

\newcommand{\lra}{\longrightarrow}

\newcommand{\CH}{{\mathrm{CH}}}
\newcommand{\AJ}{{\mathrm{AJ}}}

\newcommand{\hf}{{\mathbf{f}}}
\newcommand{\hg}{{\mathbf{g}}}
\newcommand{\hh}{{\mathbf h}}

\DeclareMathOperator{\Hom}{Hom}

\newcommand{\res}{\mathrm{res}}

\include{thebibliography}

\begin{document}

\title[Generalized Kato classes and exceptional zero conjectures]{Generalized Kato classes and exceptional zero conjectures}

\author{\'Oscar Rivero}

\begin{abstract}
The primary objective of this paper is the study of different instances of the {\em elliptic Stark conjectures} of Darmon, Lauder and Rotger, in a situation where the elliptic curve attached to the modular form $f$ has split multiplicative reduction at $p$ and the arithmetic phenomena are specially rich. For that purpose, we resort to the principle of improved $p$-adic $L$-functions and study their $\mathcal L$-invariants. We further interpret these results in terms of {\it derived} cohomology classes coming from the setting of diagonal cycles, showing that the same $\mathcal L$-invariant which arises in the theory of $p$-adic $L$-functions also governs the arithmetic of Euler systems. Thus, we can reduce, in the split multiplicative situation, the conjecture of Darmon, Lauder and Rotger to a more familiar statement about higher order derivatives of a triple product $p$-adic $L$-function at a point lying {\em inside} the region of classical interpolation, in the realm of the more well-known {\em exceptional zero conjectures}.
\end{abstract}

\address{O. R.: Departament de Matem\`{a}tiques, Universitat Polit\`{e}cnica de Catalunya, C. Jordi Girona 1-3, 08034 Barcelona, Spain}
\email{oscar.rivero@upc.edu}


\maketitle

\tableofcontents

\section{Introduction}

The elliptic Stark conjecture was first formulated by Darmon, Lauder and Rotger in \cite{DLR1} as a ``more constructive alternative to the Birch and Swinnerton-Dyer conjecture, since it often allows the efficient analytic computation of $p$-adic logarithms of global points".
The conjecture relates a $p$-adic iterated integral attached to a triple $(f,g,h)$ of cuspidal modular forms with a regulator given in terms of points in an elliptic curve, in a rank 2 situation. Until the moment, not too much work towards the proof of the conjecture has been done: most of the results are restricted to situations where there exists a factorization of $p$-adic $L$-functions, which allows to interpret the conjecture in terms of the more familiar objects of Bertolini--Darmon--Prasanna \cite{BDP}.

However, recent works of Bertolini--Seveso--Veneruci \cite{BSV}, \cite{BSV2} and Darmon--Rotger \cite{DR3.5}, \cite{DR3} suggest an alternative conjecture also in terms of triple product $p$-adic $L$-functions: while the first formulation of \cite{DLR1} is concerned with the $p$-adic value at a point lying {\em outside} the region of classical interpolation, the {\em new} version we discuss is about higher order derivatives at a point which belongs to the classical interpolation region. This setting is germane to that explored firstly by Greenberg--Stevens \cite{GS} and then by Bertolini--Darmon \cite{BD} or Venerucci \cite{Ven}. We propose an alternative conjecture in the split multiplicative setting, and one of the main results of this note is the discussion of the equivalence between both formulations, using for that purpose the setting of generalized cohomology classes. This relies, however, on an apparently deep fact about periods of weight one modular forms, stated in \cite{DR2.5} as Conjecture 2.1. We believe that this {\em translation} of the conjecture to a more well understood setting provides new evidence for a better understanding of the problem.

The genesis of this project comes from a parallel story where a new conjecture, formulated in \cite{DLR2}, arises; this gives a formula for the $p$-adic iterated integral when the modular form $f$ is no longer cuspidal, but an Eisenstein series. In \cite{RiRo1}, the authors envisaged a method of proof for this conjecture when the two modular forms $(g,h)$ are self-dual: this was based on Hida's improved factorization theorem for the Hida--Rankin $p$-adic $L$-function and allowed us to study the conjecture in terms of a question concerning Galois deformations.

The discussion of our results in this paper also leads us to the study of an exceptional vanishing of the generalized cohomology classes of \cite{DR2.5} and \cite{CH}, proposing a putative refinement in terms of some {\it derived} generalized cohomology classes.

\vskip 12pt

{\bf Setting and notations.}
Fix once for all a prime number $p \geq 3$ and three positive integers $N_f$, $N_g$, $N_h$. Let $N = \lcm(N_f,N_g,N_h)$ and assume that $p \nmid N$. Let $\chi: (\mathbb Z/N\mathbb Z)^{\times} \rightarrow \mathbb C^{\times}$ be a Dirichlet character. Let \[ f \in S_2(pN_f), \quad g \in M_1(N_g, \chi), \quad h \in M_1(N_h, \bar \chi) \] be a triple of newforms of weights $(2,1,1)$, levels $(pN_f,N_g,N_h)$ and nebentype characters $(1,\chi,\bar \chi)$, where $\bar \chi$ stands for the character obtained by composing $\chi$ with complex conjugation. Further, we denote by $V_g$ and by $V_h$ the Artin representations attached to $g$ and $h$, respectively, and write $V_{gh}:=V_g \otimes V_h$. Let $H$ be the number field cut out by this representation, and $L$ for the field over which it is defined.
To simplify the exposition, we assume that $f$ has rational Fourier coefficients and that is attached via modularity to an elliptic curve $E$ with split multiplicative reduction at $p$. Under the assumption that $(pN_f,N_gN_h)=1$, the global sign of the functional equation of $L(E,V_{gh},s)$ is $+1$. We keep this assumption from now on.

Label and order the roots of the $p$-th Hecke polynomial of $g$ as \[ X^2-a_p(g)X+\chi(p) = (X-\alpha_g)(X-\beta_g) \] and do the same for those of $h$. Let $g_\alpha(q) = g(q) -\beta_g(q^p)$ denote the $p$-stabilization of $g$ with $U_p$-eigenvalue $\alpha_g$; it is defined by the $q$-expansion $g_{\alpha}(q)=g(q)-\beta_g g(q^p)$. We want to deal with a situation of exceptional zeros, that is, where one or several of the Euler factors involved in the interpolation formula of the $p$-adic $L$-function vanish (alternatively, and as we will see later on, this can be understood in terms of the eigenvalues for the Frobenius action). This naturally splits into two different settings, namely
\begin{enumerate}
\item[(a)] the case where $\alpha_g \alpha_h = 1$ (and therefore $\beta_g \beta_h = 1$); and
\item[(b)] the case where $\alpha_g \beta_h = 1$ (and therefore $\beta_g \alpha_h = 1$).
\end{enumerate}

In both cases, if we denote the roots of the $p$-th Hecke polynomial of $g$ by $\{\alpha_g,\beta_g\}$, those of $h$ are $\{1/\alpha_g,1/\beta_g\}$. As a piece of notation, we write $h_{1/\alpha}$ and $h_{1/\beta}$ for the $p$-stabilizations of $h$ with eigenvalues $1/\alpha_g$ and $1/\beta_g$, respectively. Along this work, we refer to these settings as {\it Case (a)} and {\it Case (b)}. In the framework of Beilinson--Flach elements and Hida--Rankin $p$-adic $L$-functions, the second case has been studied in \cite{RiRo1}, and the former has been worked out in \cite[Section 5]{RiRo2}.

To prove our main results, we also need a classicality property for $g$. Hence, we assume throughout that
\begin{enumerate}
\item[(H1)] the reduction of both $V_g$ and $V_h$ modulo $p$ is irreducible (this requires the choice of integral lattices $T_g$ and $T_h$, but the fact of being irreducible or not is independent of this choice);
\item[(H2)] $g$ and $h$ are $p$-distinguished, i.e, $\alpha_g \neq \beta_g, \alpha_h \neq \beta_h \pmod{p}$; and
\item[(H3)] $V_g$ is not induced from a character of a real quadratic field in which $p$ splits.
\end{enumerate}

Enlarge $L$ if necessary so that it contains all Fourier coefficients of $g_\alpha$. As shown in \cite{DLR1}, the above hypotheses ensure that any generalized overconvergent modular form with the same generalized eigenvalues as $g_\alpha$ is classical, and hence simply a multiple of $g_\alpha$.

In order to describe our results more precisely, let $\Lambda = \Z_p[[\Z^\times_p]]$ be the Iwasawa algebra and denote by $\mathcal W = \Spf(\Lambda)$ the weight space. Hida's theory associates the following data to $f$:
\begin{itemize}
\item a finite flat extension $\Lambda_{\hf}$ of $\Lambda$, giving rise to a covering $\mathrm{w}: \mathcal W_{{\bf f}} = \Spf(\Lambda_{\hf}) \lra \cW$;
\item a family of overconvergent $p$-adic ordinary modular forms ${\bf f} \in \Lambda_{\hf}[[q]]$ specializing to $f$ at some point $x_0\in \mathcal W_{{\bf f}}$ of weight $\mathrm{w}(x_0) = 2$.
\item a representation of the absolute Galois group $G_\Q$, $\varrho_{\hf}: G_\Q \lra \GL(\mathbb V_{\hf}) \simeq \GL_2(\Lambda_{\hf})$ characterized by the property that all its classical specializations coincide with the Galois representation associated by Deligne to the corresponding specialization of the Hida family.
\end{itemize}

The same occurs with $g_{\alpha}$ and $h_{\alpha}$ thanks to the work of Bellaiche and Dimitrov \cite{BeDi} on the geometry of the eigencurve for points of weight one; we denote by $\Lambda_{\hg}$ and $\Lambda_{\hh}$ the corresponding extensions of $\Lambda$ over which the Hida families $\hg$ and $\hh$ are defined, and by $y_0 \in \mathcal W_{\hg}$, $z_0 \in \mathcal W_{\hh}$ the weight one points for which the corresponding specializations agree with $g_{\alpha}$ and $h_{\alpha}$, respectively.

For each of the settings (a) and (b) presented above, we discuss three different objects which are expected to encode arithmetic information regarding the convolution of the three Galois representations attached to the modular forms $f$, $g$ and $h$. We denote by $(x,y,z)$ a triple of points in $\mathcal W_{\hf} \times \mathcal W_{\hg} \times \mathcal W_{\hh}$, whose weights are referred as $(k,\ell,m)$.

\begin{enumerate}
\item[(i)] The cohomology classes $\kappa(f,g_{\alpha},h_{\alpha})$ studied for instance in \cite{DR2.5} and \cite{CH}, arising as the specialization at weights $(2,1,1)$ of the three-variable family $\kappa(\hf,\hg,\hh)$ constructed as the image under a $p$-adic Abel-Jacobi map of certain diagonal cycles. In general, one may construct four different classes \[ \kappa(f,g_{\alpha},h_{\alpha}), \quad \kappa(f,g_{\alpha},h_{\beta}), \quad \kappa(f,g_{\beta},h_{\alpha}), \quad \kappa(f,g_{\beta},h_{\beta}), \] one for each $p$-stabilization of $g$ and $h$. Further, when some of these classes vanish, we are lead to consider their derivatives.
\item[(ii)] The special value $\Lp^f(\hf,\hg,\hh)$ at weights $(2,1,1)$ and its derivatives. Here, $\Lp^f(\hf,\hg,\hh)$ stands for the three-variable $p$-adic $L$-function attached to three Hida families, characterized by an interpolation property regarding the classical values of the triple product $L$-function at the region where $k \geq \ell+m$. When this function vanishes at the point $(2,1,1)$, the derivatives along different directions of the weight space may encode {\it interesting} arithmetic information.
\item[(iii)] The special value $\Lp^{g_{\alpha}}(\hf,\hg,\hh)$ at weights $(2,1,1)$, denoted $\Lp^{g_{\alpha}}$. This $p$-adic $L$-function is defined in an analogue way to the previous one, but now the region of interpolation is characterized by the inequality $\ell \geq k+m$ so the point $(2,1,1)$ is outside the region of classical interpolation. Similarly, we may also take $\Lp^{h_{\alpha}}(\hf,\hg,\hh)$, whose region of interpolation concerns those points for which $m \geq k+\ell$. Observe that the first value depends on the choice of $p$-stabilizations for the weight one form $g_{\alpha}$.
\end{enumerate}

{\bf (i) Cohomology classes coming from the theory of diagonal cycles.} We begin by recalling the results concerning cohomology classes. Results of this kind had already been explored in \cite{BSV2} and \cite{DR3} when $\alpha_g \alpha_h = 1$. In that case, the cohomology class is not expected to vanish, but the numerator of the (Perrin-Riou) regulator in the reciprocity law for $\Lp^f$ does, which is coherent with the fact that the $p$-adic $L$-function $\Lp^f(f,g,h)$ is zero (this can be seen, of course, as an exceptional zero coming from the vanishing of an Euler factor).

Here we are mostly interested in the case where the denominator of the Perrin-Riou regulator in the reciprocity law for $\Lp^{g_{\alpha}}$ vanishes due to another exceptional zero phenomenon. This occurs when $\alpha_g \beta_h = 1$ and leads us to recover the ideas of \cite{Cas}, \cite{RiRo1} and \cite{R1}, where this same phenomenon was studied for Heegner points, Beilinson--Flach elements and elliptic units, respectively. In those cases, the reciprocity laws linking Euler systems and $p$-adic $L$-functions were updated to {\it derived reciprocity laws}. A different approach is taken also in \cite[Section 8]{BSV}, where the authors introduce certain {\it improved} cohomology classes, which in this case we may compare in an explicit way with appropriate {\it derived} elements.

Define the three-variable Iwasawa algebra $\Lambda_{{\bf f g h}}:= \Lambda_{{\bf f}} \hat \otimes_{\mathbb Z_p} \Lambda_{{\bf g}} \hat \otimes_{\mathbb Z_p} \Lambda_{{\bf h}}$  and the $\Lambda_{{\bf f g h}}[G_\Q]$-module \[ \mathbb V_{{\hf \hg \hh}}:= \mathbb V_{\hf} \hat \otimes_{\mathbb Z_p} \mathbb V_{{\hg}} \hat\otimes_{\mathbb Z_p} \mathbb V_{{\hh}}. \] We work with $\mathbb V_{{\bf fgh}}^{\dag}$, a certain twist of it by an appropriate power of the $\Lambda$-adic cyclotomic character defined for instance in \cite[Section 5.1]{DR3} and that is needed to satisfy the self-dual assumption.

The works \cite{BSV} and \cite{DR3} attach to $(\hf,\hg,\hh)$ a $\Lambda$-adic global cohomology class \[ \kappa({\bf f}, {\bf g},{\bf h}) \in H^1(\Q,\mathbb V_{\bf fgh}^{\dag}) \] parameterized by the triple product of the weight space $\cW_{\hf \hg \hh} := \mathcal W_{{\bf f}} \times \mathcal W_{{\bf g}} \times \mathcal W_{{\bf h}}$.

Consider the specialization of the class at weights $(x_0,y_0,z_0)$, \[ \kappa(f,g_{\alpha},h_{1/\beta}) \in H^1(\mathbb Q, V_{fgh}), \]
where $V_{fgh}$ is the tensor product $V_f \otimes V_g \otimes V_h$ of the Galois representations attached to the modular forms $f$, $g$ and $h$.
This class can be shown to be trivial and hence we are placed to work with an appropriate {\it derived} class $\kappa'(f,g_{\alpha},h_{1/\beta})$.

As it occurred in the setting of Beilinson--Flach classes, the notion of derivative is rather flexible. Following \cite{RiRo1}, we consider here a derivative along an analytic direction, and keeping fixed the weight of $h$. Rather informally, this may be thought as the line $(\ell+1,\ell,1)$ of the weight space. Note that at least in the self-dual case, where we may argue that the corresponding class vanishes all along the line $(2,\ell,\ell)$, we may consider the derivative along any direction of the weight space.

Let $\alpha_{\hf}$ (resp. $\alpha_{\hg}$, $\alpha_{\hh}$) stand for the Iwasawa function corresponding to the root of the $p$-th Hecke polynomial of $\hf$ (resp. $\hg$, $\hh$) with smallest $p$-adic valuation. As an additional piece of notation, let
\begin{equation}\label{l1}
\mathcal L:=\frac{\alpha_g'}{\alpha_g} - \frac{\alpha_f'}{\alpha_f},
\end{equation}
where $\alpha_f'$ (resp. $\alpha_g'$, $\alpha_h'$) stands for the derivative of the Frobenius eigenvalues at $x_0$ (resp. $y_0$, $z_0$) when seen as an Iwasawa function along the Hida family $\Lambda_{\hf}$ (resp. $\Lambda_{\hg}$, $\Lambda_{\hh}$). Observe that we can give explicit formulas for $\mathcal L$, involving both some units and $p$-units in the field cut out by the representation $V_{gh}$ and the Tate uniformizer of the elliptic curve $E$. Hence, the $\mathcal L$-invariant governing the arithmetic of the triple $(f,g,h)$ is related both with the $\mathcal L$-invariant of the elliptic curve (the logarithm of the Tate uniformizer) and also with the regulator attached to the adjoint representation $\ad^0(V_g)$, expressed in \cite{RiRo1} as a combination of logarithms of units and $p$-units. Compare for instance this result with the main theorem of \cite{Cas}, where he interprets the $\mathcal L$-invariant attached to a modular form $f$ and an anticyclotomic character as the sum of the two $\mathcal L$-invariants. Our first main result is the following (see Theorem \ref{reclaw} for the precise formulation), relating an appropriate logarithm of the {\it derived} class with the special value $\Lp^{g_{\alpha}}$.

\begin{theorem}
The derived cohomology class satisfies \[ \langle \log_{\BK}(\kappa_p'(f,g_{\alpha},h_{1/\beta})^g), \eta_f \otimes \omega_{g_{\alpha}} \otimes \omega_{h_{1/\beta}} \rangle =  \mathcal L \cdot \Lp^{g_{\alpha}} (\hf,\hg,\hh)(x_0,y_0,z_0) \pmod{L^{\times}}, \] where the superindex $g$ stands for an appropriate projection of $\kappa_p'$ that we later introduce, and $\log_{\BK}$ refers to the Bloch--Kato logarithm, followed by the pairing $\langle \, -,- \, \rangle$ with certain canonical differentials.
\end{theorem}

\begin{rem*}
In \cite{BSV} the authors take a different approach to this exceptional zero phenomenon, and construct an {\it improved} cohomology class $\kappa_g^*(f,g_{\alpha},h_{1/\beta})$. As we will later show, there is a connection between both constructions and one may prove (under mild conditions!) that the following equality holds in $H^1(\mathbb Q,(V_f \otimes V_g \otimes V_h)_{|\mathcal S})$, where $\mathcal S$ stands for the subvariety of the weight space along which the derived and the improved class are defined, corresponding to the set of weights $k+m=\ell+2$:
\begin{equation}
\kappa'(f,g_{\alpha},h_{1/\beta}) = \mathcal L \cdot \kappa_g^*(f,g_{\alpha},h_{1/\beta}).
\end{equation}
\end{rem*}

{\bf (ii) The special value $\Lp^f$ and derivatives of the triple product $p$-adic $L$-function.} In subsequent parts of the article we use the previous cohomology classes to study different instances of the elliptic Stark conjecture. Section 4 is devoted to analyze higher order derivatives of $\Lp^f(\hf,\hg,\hh)$ at $(x_0,y_0,z_0)$. The presence of an Euler factor which vanishes at weights $(2,1,1)$ automatically forces the vanishing of that value. Therefore, it is natural to formulate several conjectures for the value of the derivatives of $\Lp^f(\hf,\hg,\hh)$.

When $\alpha_g \alpha_h = 1$ and $L(f \otimes g \otimes h,1) \neq 0$, the results of \cite{BSV} relying on the existence of an improved $p$-adic $L$-function allow us to state the following result. Although this can be seen as a straightforward corollary of the results developed in loc.\,cit., we want to point out that the $\mathcal L$-invariants attached to both $g$ and $h$ have a strong connection with the arithmetic of number fields. This reveals that in the rank 0 situation the quantity $\Lp^f$ is also a putative refinement of the more well-known $\mathcal L$-invariants of Greenberg--Stevens, where not only the Tate period $q_E$ appears. This result follows from \cite[Proposition 8.2]{BSV}.

\begin{pro*}[Bertolini--Seveso--Venerucci]
Let $I$ denote the ideal of functions in $\Lambda_{\hf \hg \hh}$ which vanish at $(x_0,y_0,z_0)$. Assume that $L(f \otimes g \otimes h, 1) \neq 0$, and let $\mathcal L_{\xi}:=\alpha_{\xi}'/\alpha_{\xi}$, for $\xi \in \{f,g_{\alpha},h_{\alpha}\}$. Then, up to a constant in $L^{\times}$, \[ \Lp^f(\hf,\hg,\hh) = (\mathcal L_{g_{\alpha}}-\mathcal L_f)(\ell-1) + (\mathcal L_{h_{\alpha}}-\mathcal L_f)(m-1) \pmod{I^2}. \] Moreover, the quantities $\mathcal L_{\chi}$ are explicitly computable in terms of the arithmetic of number fields and elliptic curves.
\end{pro*}

Observe for example that the derivative along the $y$-direction agrees with the $\mathcal L$-invariant that also arises as the derivative of the diagonal class discussed before.

However, the most interesting case appears when $L(f \otimes g \otimes h,1) = 0$. Let us put ourselves in the setting of \cite{DLR1} and assume that $(E(H) \otimes V_{gh}^{\vee})^{\Gal(H/\mathbb Q)}$ is two-dimensional, where $V_{gh}^{\vee}$ stands for the contragradient representation of $V_{gh}$. This group is equipped with an inclusion in the $p$-adic Selmer group corresponding to the group of extensions of $\mathbb Q_p$ by $V_{fgh}$ in the category of $\mathbb Q_p$-linear representations of $G_{\mathbb Q}$ that are crystalline at $p$. This group is denoted by $H_{\fin}^1(\mathbb Q, V_{fgh})$, and we also assume that is two-dimensional (the latter would follow from the Birch and Swinnerton-Dyer conjecture for the pair $(E,V_{gh})$ and the finiteness of the corresponding Tate--Shafarevich group).

Let $\{P,Q\}$ denote generators of $(E(H) \otimes V_{gh}^{\vee})^{G_{\mathbb Q}}$, and fix a basis $\{e_{\alpha \alpha}^{\vee}, e_{\alpha \beta}^{\vee}, e_{\beta \alpha}^{\vee}, e_{\beta \beta}^{\vee}\}$ of $V_{gh}^{\vee}$ as a $G_{\mathbb Q_p}$-module with the Frobenius action. This allows us to write \[ P = P_{\beta \beta} \otimes e_{\beta \beta}^{\vee} + P_{\beta \alpha} \otimes e_{\beta \alpha}^{\vee} + P_{\alpha \beta} \otimes e_{\alpha \beta}^{\vee} + P_{\alpha \alpha} \otimes e_{\alpha \alpha}^{\vee}, \] and similarly for $Q$. Here, the arithmetic Frobenius $\Frob_p$ acts on $P_{\beta \beta}$ with eigenvalue $\beta_g \beta_h$ and analogously for the remaining components. In this scenario, we can conjecture the following result, that we extensively discuss in Section \ref{conj-S2}.

\begin{conj}\label{conj1}
Assume that the $L$-dimension of $(E(H) \otimes V_{gh}^{\vee})^{G_{\mathbb Q}}$ is two. Then, under the running assumptions, the $p$-adic $L$-function $\Lp^f(\hf,\hg,\hh)$ satisfies
\[ \frac{\partial^2 \Lp^f(\hf_x,g_{\alpha},h_{1/\alpha})}{\partial x^2} \Big|_{x=x_0} = \log_p(P_{\beta \beta}) \cdot \log_p(Q_{\alpha \alpha}) - \log_p(Q_{\beta \beta}) \cdot \log_p(P_{\alpha \alpha}) \pmod{L^{\times}}. \] If the $L$-dimension of $(E(H) \otimes V_{gh}^{\vee})^{G_{\mathbb Q}}$ is greater than two, then the left hand side vanishes.
\end{conj}

There are other interesting lines along weight space to take derivatives. For example, in \cite{CH} the study is concerned with the line $(2,\ell,\ell)$, where the derivatives are connected with appropriate {\it derived} heights of the points $P$ and $Q$.

The work of Bertolini--Seveso--Venerucci and Darmon--Rotger establishes the conjecture for the case where $g$ and $h$ are theta series of a quadratic field where $p$ is inert, which leads to a decomposition $V_{gh} = V_{\psi_1} \oplus V_{\psi_2}$. In the imaginary case, we can extend their computations to the split case, observing that here one has a trivial equality of the form $0=0$.

Therefore, we may establish that Conjecture \ref{conj1} holds in some dihedral cases. The first part of this Proposition follows from \cite[Theorem A]{BSV2}, and the second is established as part of Proposition \ref{redu}.

\begin{propo}
Conjecture \ref{conj1} holds in the following cases:
\begin{enumerate}
\item[(a)] CM or RM series with $p$ inert in $K$ and at least one of $\psi_1$ or $\psi_2$ being a genus characters;
\item[(b)] CM series with $p$ split in $K$, $V_{gh} = V_{\psi_1} \oplus V_{\psi_2}$, with each component of rank one.
\end{enumerate}
\end{propo}
We must say that in all these cases the proof is based on a factorization formula, so we expect that new ideas would be required for the proof in the general case.

{\bf (iii) The special value $\Lp^{g_{\alpha}}$.} In the last section, we discuss a way to connect the previous conjecture with the elliptic Stark conjecture of \cite{DLR1} when $\alpha_g \alpha_h = 1$. Recall that the conjecture predicts that
\begin{equation}\label{Lpg}
\Lp^{g_{\alpha}}(f,g_{\alpha},h_{\alpha}) = \frac{\log_p(P_{\beta \alpha})\log_p(Q_{\beta \beta})-\log_p(P_{\beta \beta})\log_p(Q_{\beta \alpha})}{\log_p(u_{g_{\alpha}})} \pmod{L^{\times}},
\end{equation}
with $u_{g_{\alpha}}$ being a Gross--Stark unit whose characterization we later recall.
In particular, it is expected that this unit could be expressed as a ratio of periods attached to weight one forms. These two periods, denoted by $\Omega_{g_{\alpha}}$ and $\Xi_{g_{\alpha}}$, will play a prominent role in the last part of the work. More precisely, in \cite[eq.\,(9)]{DR2.5}, the authors introduce a $p$-adic period, $\mathcal L_{g_{\alpha}}=\Omega_{g_{\alpha}}/\Xi_{g_{\alpha}}$ and conjecture (see Conjecture 2.1 of loc.\, cit.)
\begin{equation}\label{peri}
\mathcal L_{g_{\alpha}} = \log_p(u_{g_{\alpha}}).
\end{equation}

In Section 5 we consider the following three conjectures:
\begin{enumerate}
\item [(i)] the elliptic Stark conjecture for $\Lp^{g_{\alpha}}$;
\item [(ii)] the conjecture for the second derivative along the $f$-direction for $\Lp^f$, i.e., Conjecture \ref{conj1};
\item[(iii)] \cite[Conjecture 2.1]{DR2.5} about periods of weight one modular forms. Proposition \ref{evidence} can be seen as an extra piece of theoretical evidence towards this conjecture, showing that \[ \frac{\mathcal L_{g_{\alpha}}}{\mathcal L_{g_{\beta}}} = \frac{\log_p(u_{g_{\alpha}})}{\log_p(u_{g_{\beta}})}. \]
\end{enumerate}

Under certain non-vanishing hypothesis, we prove that if two of the previous conjectures are true, the third one automatically holds. In particular, we establish the following in Corollary \ref{perintro}.

\begin{theorem}
Let $g$ and $h$ be theta series of a quadratic field (either real or imaginary) where $p$ is inert. Write $V_{gh} = V_{\psi_1} \oplus V_{\psi_2}$, and assume that either $\psi_1$ or $\psi_2$ is a genus character. Then, under the given assumptions, the equality \eqref{peri} is equivalent to the elliptic Stark conjecture of Darmon, Lauder and Rotger \eqref{Lpg}.
\end{theorem}

All the previous results are based on the interaction of the different arithmetic actors when $\alpha_g \alpha_h = 1$. The case where $\alpha_g \beta_h = 1$ is more subtle, since here the cohomology class $\kappa(f,g_{\alpha},h_{1/\beta})$ vanishes and we cannot extract the same arithmetic information. In any case, we expect that a similar result must hold in this setting. The reason is that the value of $\Lp^{g_{\alpha}}$ does not depend on the choice of a $p$-stabilization for $h$, and hence we can also give a conjectural expression for the derived cohomology class in terms of points, in complete analogy with \cite[Theorem B]{RiRo1}.

\begin{conj}
The following equality holds in $H_{\fin}^1(\mathbb Q, V_{fgh})$:
\begin{equation}
\kappa'(f,g_{\alpha},h_{1/\beta}) = \frac{\mathcal L}{\Xi_{g_{\alpha}} \cdot \Omega_{h_{1/\beta}}} \cdot \frac{\log_p(P_{\beta \beta}) \cdot Q - \log_p(Q_{\beta \beta}) \cdot P}{\log_p(u_{g_{\alpha}})} \pmod{L^{\times}}.
\end{equation}
\end{conj}
Proceeding as in \cite{DR2.5} and \cite{RiRo2} we may also obtain expressions (at least conjecturally) for the three remaining cohomology classes, $\kappa'(f,g_{\beta},h_{1/\alpha})$, $\kappa(f,g_{\alpha},h_{1/\alpha})$ and $\kappa(f,g_{\beta},h_{1/\beta})$.

\vskip 12pt

{\bf Acknowledgements.}  It is a pleasure to thank Victor Rotger for his interest on this project and for a careful reading of the manuscript. I am also indebted to Daniel Barrera, Henri Darmon, and Giovanni Rosso for several enlightening conversations around this work. I sincerely thank the anonymous referee for a careful reading of the text, whose comments notably contributed to improve the exposition of this article.

The author has been supported by the European Research Council (ERC) under the European Union's Horizon 2020 research and innovation programme (grant agreement No. 682152). The author has also received financial support through ``la Caixa" Banking Foundation (grant LCF/BQ/ES17/11600010).

\section{Preliminaries}

This section aims to give an overview of the setting we present, concerned with triple product $p$-adic $L$-functions, and also recalls some known results in other related scenarios coming from the theory of elliptic curves and weight one modular forms.

\subsection{Hsieh's triple product p-adic L-function}

Fix an algebraic closure $\bar{\mathbb Q}$ of $\mathbb Q$. For a number field $K$, let $G_K := \Gal(\bar{\mathbb Q}/K)$ denote its absolute Galois group. Fix also an odd prime $p$ and an embedding $\bar{\mathbb Q} \hookrightarrow \bar{\mathbb Q}_p$.

The formal spectrum $\mathcal W = \Spf(\Lambda)$ of the Iwasawa algebra $\Lambda = \mathbb Z_p[[\mathbb Z_p^{\times}]]$ is called the weight space attached to $\Lambda$. The weight space is equipped with a distinguished class of {\it arithmetic points} $\nu_{s,\varepsilon}$ indexed by integers $s \in \mathbb Z$ and Dirichlet characters $\varepsilon: (\mathbb Z/p^r \mathbb Z)^{\times} \rightarrow \bar{\mathbb Q}^{\times}$ of $p$-power conductor. The point $\nu_{s,\varepsilon} \in \mathcal W$ is defined by the rule \[ \nu_{s,\varepsilon}(n)=\varepsilon(n)n^s. \]

Let $(\hf,\hg,\hh)$ be a triple of $p$-adic Hida families of tame levels $N_{\hf}$, $N_{\hg}$, $N_{\hh}$ and tame characters $\chi_{\hf}$, $\chi_{\hg}$, $\chi_{\hh}$. Let also $(\hf^*,\hg^*,\hh^*)$ denote the conjugate triple, and assume that $\chi_{\hf} \chi_{\hg} \chi_{\hh}=1$ (this is referred to as the self-duality assumption). Set $N = \lcm(N_{\hf},N_{\hg},N_{\hh})$, and suppose that $p \nmid N$.

Let $\Lambda_{\hf}$, $\Lambda_{\hg}$ and $\Lambda_{\hh}$ be the finite extensions of $\Lambda$ generated by the coefficients of the Hida families $\hf$, $\hg$ and $\hh$, respectively. The weight space attached to $\Lambda_{\hf}$ is $\mathcal W_{\hf} := \Spf(\Lambda_{\hf})$. Since $\Lambda_{\hf}$ is a finite flat algebra over $\Lambda$, there is a natural finite map \[ \pi: \mathcal W_{\hf} := \mathrm{Spf}(\mathcal W_{\hf})\stackrel{\mathrm{w}}{\lra} \cW, \] and we say that a point $x \in \mathcal W_{\hf}$ is arithmetic of weight $s$ and character $\varepsilon$ if $\pi(x) = \nu_{s,\varepsilon}$.

A point $x \in \mathcal W_{\hf}$ of weight $k \geq 1$ and character $\varepsilon$ is said to be crystalline if $\varepsilon = 1$ and there exists an eigenform $f_x^{\circ}$ of level $N$ such that $f_x$ is the ordinary $p$-stablization of $f_x^{\circ}$. We denote by $\mathcal W_{\hf}^{\circ}$ the set of crystalline arithmetic points of $\mathcal W_{\hf}$.

Finally, set $\Lambda_{{\bf fgh}}=\Lambda_{\hf} \hat \otimes \Lambda_{\hg} \hat \otimes \Lambda_{\hh}$ and let $\mathcal W_{{\bf fgh}}^{\circ} := \mathcal W_{\hf}^{\circ} \times \mathcal W_{\hg}^{\circ} \times \mathcal W_{\hh}^{\circ} \subset \mathcal W_{{\bf fgh}} = \Spf(\Lambda_{{\bf fgh}})$ be the set of triples of crystalline classical points, at which the three Hida families specialize to modular forms with trivial nebentype at $p$. This set admits the natural partition \[ \mathcal W_{{\bf fgh}}^{\circ} = \mathcal W_{{\bf fgh}}^{f} \sqcup \mathcal W_{{\bf fgh}}^{g} \sqcup \mathcal W_{{\bf fgh}}^{h} \sqcup \mathcal W_{{\bf fgh}}^{\bal}, \]
where
\begin{itemize}
\item $\mathcal W_{{\bf fgh}}^f$ denotes the set of points $(x,y,z) \in \mathcal W_{{\bf fgh}}^{\circ}$ of weights $(k,\ell,m)$ such that $k \geq \ell+m$.
\item $\mathcal W_{{\bf fgh}}^g$ and $\mathcal W_{{\bf fgh}}^h$ are defined similarly, replacing the role of $\hf$ by $\hg$ (resp. $\hh$).
\item $\mathcal W_{{\bf fgh}}^{\bal}$ is the set of balanced triples, consisting of points $(x,y,z)$ of weights $(k,\ell,m)$ such that each of the weights is strictly smaller than the sum of the other two.
\end{itemize}


Recall from \cite[Section 1.4]{DR3} the notion of test vector. As proved in Section 3.5 of loc.\,cit. following \cite{Hs}, there is a canonical choice of test vectors for which there exists a {\it square-root} $p$-adic $L$-function \[ \Lp^f(\hf, \hg, \hh): \mathcal W_{{\bf fgh}} \rightarrow \mathbb C_p, \] characterized by an interpolation property relating its values at classical points $(x,y,z) \in \mathcal W_{{\bf fgh}}^f$ to the square root of the central critical value of Garrett's triple-product complex $L$-function $L(\hf_x,\hg_y,\hh_z,s)$ associated to the triple of classical eigenforms $(\hf_x,\hg_y,\hh_z)$. For the following proposition, let $\alpha_{\hf_x}$ and $\beta_{\hf_x}$ be the roots of the $p$-th Hecke polynomial of $\hf_x$, ordered in such a way that $\ord_p(\alpha_{\hf_x}) \leq \ord_p(\beta_{\hf_x})$. The following result is \cite[Proposition 5.1]{DR3}.

\begin{propo}\label{3l}
Fix test vectors $(\tilde{\hf},\tilde{\hg},\tilde{\hh})$ as in \cite[Section 3]{Hs}. Then $\Lp^f(\tilde{\hf},\tilde{\hg},\tilde{\hh})$ lies in $\Lambda_{{\bf fgh}}$ and for every $(x,y,z) \in \mathcal W_{{\bf fgh}}^f$ of weights $(k,\ell,m)$ we have \[ \Lp^f(\tilde{\hf},\tilde{\hg},\tilde{\hh})^2(x,y,z) = \frac{\mathfrak a(k,\ell,m)}{\langle \hf_x^{\circ},\hf_x^{\circ} \rangle^2} \cdot \mathfrak e^2(x,y,z) \times L(\hf_x^{\circ},\hg_y^{\circ},\hh_z^{\circ},c), \] where
\begin{enumerate}
\item $c = \frac{k+\ell+m-2}{2}$.
\item $\mathfrak a(k,\ell,m) = (2\pi i)^{-2k} \cdot \Big( \frac{k+\ell+m-4}{2} \Big) ! \cdot \Big( \frac{k+\ell-m-2}{2} \Big) ! \cdot \Big( \frac{k-\ell+m-2}{2} \Big) ! \cdot \Big( \frac{k-\ell-m}{2} \Big) !$,
\item $\mathfrak e(x,y,z) = \mathcal E(x,y,z)/\mathcal E_0(x) \mathcal E_1(x)$ with
    \begin{eqnarray}
      \nonumber \mathcal E_0(x) &:=& 1-\chi_f^{-1}(p) \beta_{\hf_x}^2p^{1-k}, \\
      \nonumber \mathcal E_1(x) &:=& 1-\chi_f(p)\alpha_{\hf_x}^{-2}p^{k-2}, \\
      \nonumber \mathcal E(x,y,z) &:=& \Big(1-\chi_f(p) \alpha_{\hf_x}^{-1} \alpha_{\hg_y}\alpha_{\hh_z}p^{\frac{k-\ell-m}{2}} \Big) \times \Big(1-\chi_f(p) \alpha_{\hf_x}^{-1} \alpha_{\hg_y}\beta_{\hh_z}p^{\frac{k-\ell-m}{2}} \Big)\\ \nonumber
       & & \times \Big(1-\chi_f(p) \alpha_{\hf_x}^{-1} \beta_{\hg_y}\alpha_{\hh_z}p^{\frac{k-\ell-m}{2}} \Big) \times \Big(1-\chi_f(p) \alpha_{\hf_x}^{-1} \beta_{\hg_y}\beta_{\hh_z}p^{\frac{k-\ell-m}{2}} \Big).
    \end{eqnarray}
\end{enumerate}
\end{propo}

For simplicity, we fix once and for all these test vectors and remove their dependence in the notation.

There is an ostensible parallelism between this $p$-adic $L$-function and the so-called Hida--Rankin $p$-adic $L$-function attached to a pair of Hida families $(\hg,\hh)$, but where the cyclotomic variable $s$ is allowed to move freely. It may be instructive to keep in mind this analogy for the subsequent results.

Some of the easiest cases to understand these triple product $p$-adic $L$-functions arise when the representation attached to $V_{gh}$ is irreducible. In particular, assume that $g$ is a weight one theta series attached to a quadratic field $K$ (either real of imaginary) where $p$ remains inert. Then, $V_{gh} = V_{\psi_1} \oplus V_{\psi_2}$, and under the assumption that at least one between $\psi_1$ or $\psi_2$ is a genus (quadratic) character, the works \cite{BSV2} and \cite{DR3} show that
\begin{equation}
\Lp^f(\hf,g,h)^2 = \mathfrak f(k) \cdot L_p(\hf/K, \psi_1) \cdot L_p(\hf/K,\psi_2),
\end{equation}
where $\mathfrak f(k)$ is a bounded analytic function on $\Lambda_{\hf}$ such that $\mathfrak f(x_0) \in L^{\times}$.
Here, $\Lp(\hf/K,\psi)$ is the two-variable $p$-adic $L$-function attached to a Hida family $\hf$ and a character $\psi$ of a quadratic field.

As a word of caution, observe that there are {\it three} different $p$-adic $L$-functions, depending on the region of classical interpolation (associated to the dominant weight).

\subsection{Improved p-adic L-functions}

It is a natural phenomenon in the study of $p$-adic $L$-functions that some of the Euler factors arising in the interpolation process are analytic along a subvariety of the weight space. When this happens, one is tempted to define {\it improved} $p$-adic $L$-functions, that is, functions over the corresponding subvariety characterized by the same interpolation property, but with these Euler factors removed. This is a quite well-known phenomenon, which dates back to Greenberg--Stevens \cite{GS} and their study of the Mazur--Kitagawa $p$-adic $L$-function. This was one of the key ingredients in the proof of our main results in \cite{RiRo1} and we would like to stress the limitations of the method in this triple product setting. The interest of this study is that we also want to discuss later on its applicability from the Euler system side in order to construct {\it improved cohomology classes}.

For the sake of simplicity, assume that $\chi_f$ is trivial. In the setting of triple product $p$-adic $L$-functions we have just discussed, one of the Euler factors appearing in the interpolation property of $\Lp^f(\hf,\hg,\hh)$ is \[ 1-\frac{\alpha_{\hg_y} \alpha_{\hh_z}}{\alpha_{\hf_x}}p^{\frac{k-\ell-m}{2}}, \] which is an Iwasawa function along the surface $\mathcal S_{k=\ell+m}$ defined by \[ \mathcal S_{k=\ell+m}= \{ (x,y,z) \in \mathcal W_{\hf}^{\circ} \times \mathcal W_{\hg}^{\circ} \times \mathcal W_{\hh}^{\circ} \text{ such that } k = \ell+m \}. \]

The definitions given in \cite[Def.\,4.4]{DR1} can be adapted to yield an {\em improved} $p$-adic $L$-function $\Lp(\hf,\hg,\hh)^*$ on $\mathcal S_{k=\ell+m}$, by replacing the family $\hh \times d^t \hg^{[p]}$ with the family $\hh \times \hg$, whose coefficients vary analytically because $t=0$ on $\mathcal S_{k=\ell+m}$.

\begin{propo}\label{propo-imp}
There exists an analytic $p$-adic $L$-function over the surface $\mathcal S_{k=\ell+m}$, denoted by $\Lp^f(\hf,\hg,\hh)^*$, such that the following equality holds in $\mathcal S_{k=\ell+m}$:
\begin{equation}\label{eq-imp}
\Lp^f(\hf,\hg,\hh) = (1-\alpha_{\hf}^{-1} \alpha_{\hg} \alpha_{\hh}) \Lp^f(\hf,\hg,\hh)^*.
\end{equation}
\end{propo}

\begin{proof}
This follows from the proof of \cite[Proposition 8.2]{BSV} and the discussion after it.
\end{proof}

We point out that the improved $p$-adic $L$-function we have considered, $\Lp^f(\hf,\hg,\hh)^*$, interpolates classical $L$-values, in the same way than $\Lp^f(\hf,\hg,\hh)$, but with the vanishing Euler factor removed. Therefore, its value at $(x_0,y_0,y_0)$ is given by an explicit non-zero multiple of the square root of the algebraic part of $L(f,g,h,1)$. In particular, $L(f,g,h,1) \neq 0$ if and only if the improved $p$-adic $L$-function does not vanish at $(x_0,y_0,y_0)$.

Observe however that we may also consider other Euler factors. Take for example \[ 1-\frac{\bar \chi_h(p) \alpha_{\hh_z}}{\alpha_{\hg_y}\alpha_{\hf_x}}p^{\frac{k+\ell-m-2}{2}}, \] which is analytic along $k+\ell=m+2$.

We would expect that one can establish that these factors (each along its respective region) divide the $p$-adic $L$-function and yield other improved $p$-adic $L$-functions satisfying {\it mild} analytic and interpolation properties.

\subsection{Exceptional zeros and L-invariants}

The situations we study in this article are mostly concerned with the so-called exceptional zero phenomenon. We now recall several results which appear in the literature around that, mainly in \cite{GS}, \cite{Ven} and \cite{RiRo1}. As anticipated before, the point is that the $\mathcal L$-invariant governing the arithmetic of $V_f \otimes V_{gh}$ is a combination of the $\mathcal L$-invariants attached to $f$ and the adjoint of $g$ (or the adjoint of $h$, according to the {\it direction} we choose). For a brief summary of the usual definition and the main properties of the adjoint representation in this scenario, see \cite[Section 1.2]{DLR1}.

The aim of this section is to give an arithmetic description of the different $\mathcal L$-invariants that later appear in the setting of triple products, to have a complete description of our picture.

Let us define, to ease notations,
\begin{equation}
\mathcal L(\ad^0(g_{\alpha})) := \frac{\alpha_g'}{\alpha_g}, \qquad \mathcal L(E) := \frac{\alpha_f'}{\alpha_f},
\end{equation}
where the derivative is taken along the unique Hida family passing through $g_{\alpha}$ and $f$, respectively, and then evaluating at the points corresponding to $g_{\alpha}$ and $f$.

{\bf I. The $\mathcal L$-invariant of the adjoint of a modular form.} One of the main results of \cite{RiRo1} was the computation of the $\mathcal L$-invariant associated to the adjoint of a modular form.

As shown in \cite[Lemma 1.1]{DLR2}, we have \[ \dim_L (\mathcal O_H^\times \otimes \ad^0(g))^{G_{\mathbb Q}} = 1, \quad \dim_L (\mathcal O_H[1/p]^\times/p^{\mathbb Z} \otimes \ad^0(g))^{G_{\mathbb Q}} = 2. \]
Fix a generator $u$ of $ (\mathcal O_H^\times \otimes \ad^0(g))^{G_{\mathbb Q}}$ and also an element $v$ of $(\mathcal O_H^{\times}[1/p]^{\times} \otimes \ad^0(g))^{G_{\mathbb Q}}$ such that $\{ u, v\}$ is a basis of $(\mathcal O_H[1/p]^\times \otimes \ad^0(g))^{G_{\mathbb Q}}$. The element $v$ may be chosen to have $p$-adic valuation $\ord_p (v)=1$, and we do so. Viewed as a $G_{\Q_p}$-module, $\ad^0(g)$ decomposes as $\ad^0(g) = L^1 \oplus L^{\alpha/\beta} \oplus L^{\beta/\alpha}$, where all the summands are 1-dimensional subspaces characterized by the property that the arithmetic Frobenius $\Frob_p$ acts on it with eigenvalue $1$, $\alpha/\beta$ and $\beta/\alpha$, respectively. Let $H_p$ denote the completion of $H$ in $\bar{\mathbb Q}_p$  and let
\[ u_1, \,\, u_{\alpha/\beta},  \,\,  u_{\beta/\alpha},  \,\,  v_1,  \,\,  v_{\alpha/\beta},  \,\,  v_{\beta/\alpha} \in H_p^\times \otimes_{\mathbb Q} L \quad \pmod{L^{\times}} \]
denote the projection of the elements $u$ and $v$ in $(H_p^\times \otimes \ad^0(g))^{G_{\mathbb Q_p}}$ to the above lines. By construction we have $u_1, v_1 \in \Q_p^\times$ and
\[ \Frob_p(u_{\alpha/\beta}) = \frac{\beta}{\alpha} u_{\alpha/\beta}, \quad \Frob_p(v_{\alpha/\beta}) = \frac{\beta}{\alpha} v_{\alpha/\beta}, \quad \Frob_p(u_{\beta/\alpha}) = \frac{\alpha}{\beta} u_{\beta/\alpha}, \quad \Frob_p(v_{\beta/\alpha}) = \frac{\alpha}{\beta} v_{\beta/\alpha}. \]

Let \[ \log_p: H_p^\times \otimes L \lra H_p \otimes L \]
denote the usual $p$-adic logarithm.

Then, one the main results of \cite{RiRo1} was the computation of $\mathcal L(\ad^0(g_{\alpha}))$, which can be expressed as
\begin{equation}\label{inv2}
 \mathcal L(\ad^0(g_{\alpha})) = - \frac{\log_p(v_1)\log_p(u_{\alpha/\beta})-\log_p(u_1)\log_p(v_{\alpha/\beta})}{2\ord_p(v_1) \cdot \log_p(u_{\alpha/\beta})}.
\end{equation}

{\bf II. The $\mathcal L$-invariant of an elliptic curve (rank 0).} In \cite{GS}, the authors prove a conjecture of Mazur, Tate and Teitelbaum \cite{MTT} expressing the quantity $L_p(E,1)$ in terms of the derivative of $L(E,1)$ when the rank is zero. As a consequence of this, they show that an elliptic curve with split multiplicative reduction at $p$ satisfies
\begin{equation}\label{inv1}
\mathcal L(E) = -\frac{\log_p(q_E)}{2\ord_p(q_E)},
\end{equation}
where $q_E$ is Tate's uniformizer for the elliptic curve $E$. We write $L_p(\hf)(x,s)$ for the usual two-variable Mazur--Kitagawa $p$-adic $L$-function, and $x_0$ for the weight two point satisfying $\hf_{x_0}=f$, with $f$ the modular form attached to $E$ by modularity.

As recalled for instance in the discussion of \cite[Remark 1.13]{BD}, there exists an improved $p$-adic $L$-function along $s=1$, that we denote here as $L_p^*(\hf)(x)$ and which is characterized by \[ L_p(\hf)(x,1) = (1-a_p(\hf_x)^{-1}) \cdot L_p^*(\hf)(x). \]
Observe that in a rank 0 situation $L_p^*(\hf_{x_0})$ is a non-zero algebraic number which agrees (up to constant) with the algebraic part of the classical $L$-value.

{\bf III. The $\mathcal L$-invariant of an elliptic curve (rank 1).} In a rank 1 situation, Venerucci relates the second derivatives of the Mazur--Kitagawa $p$-adic $L$-function with certain heights of Heegner points. Observe that in this setting, $L_p(\hf)(x_0,1) = 0$ and the same happens for its first derivatives. To determine the second order derivatives, he recasts in \cite{Ven} to the theory of Selmer complexes and Nekov\'a\v{r}'s Selmer groups, as introduced in \cite{Nek}.

Following the conventions used in loc.\,cit., let $\tilde H_{\fin}^1$ be Nekovar's extended Selmer group. It is a $\mathbb Q_p$-module, equipped with a natural inclusion of the extended Mordell-Weil group of $E$, that we denote by $E^{\dag}(\mathbb Q) \otimes \mathbb Q_p$. In general, \[ \tilde H_{\fin}^1(\mathbb Q, V_p(E)) = H_{\fin}^1(\mathbb Q, V_p(E)) \oplus \mathbb Q_p \cdot q_E, \] where $H_{\fin}^1(\mathbb Q, V_p(E))$ is the Bloch--Kato $p$-adic Selmer group. Using Nekovar and Venerucci's results, there is a canonical $\mathbb Q_p$-bilinear form \[ \langle \cdot,\cdot \rangle: \tilde H_{\fin}^1(\mathbb Q, V_p(E)) \otimes_{\mathbb Q_p} \tilde H_{\fin}^1(\mathbb Q, V_p(E)) \rightarrow I/I^2, \] where $I$ stands for the augmentation ideal of the cyclotomic Iwasawa algebra, and which may be thought as the ring of functions vanishing at the point $(x,s)=(x_0,1)$, that with a slight abuse of notation we denote by $(2,1)$. This is the so-called {\em height-weight} pairing, which decomposes as \[ \langle \cdot,\cdot \rangle = \langle \cdot,\cdot \rangle_p^{\cyc} \cdot \{s-1\} + \langle \cdot,\cdot \rangle_p^{\wt} \cdot \{k-2\}, \] where $\langle \cdot,\cdot \rangle_p^{\cyc}$ and $\langle \cdot,\cdot \rangle_p^{\wt}$ are canonical $\mathbb Q_p$-valued pairings on the extended Selmer group. Finally, the Schneider height is defined by \[ \langle \cdot,\cdot \rangle_p^{\Sch} = \langle x,y \rangle_p^{\cyc} - \frac{\log_p(\res_p(x)) \cdot \log_p(\res_p(y))}{\log_p(q_E)}, \] where $\res_p(x)$ is the localization-at-$p$ map. The following result provides expressions for the second derivative of $L_p(\hf)$ along different directions of the weight space.

\begin{propo}\label{3ven}
The following formulas hold, where $P$ is a generator of the Mordell--Weil group $E(\mathbb Q)$.
\begin{enumerate}
\item[(a)] \[ \frac{d^2 L_p(\hf)(k,k/2)}{dk^2} \Big|_{k=2} = \log_E(P)^2 \pmod{L^{\times}}; \]
\item[(b)] \[ \frac{d^2 L_p(\hf)(k,1)}{dk^2} \Big|_{k=2} = \mathcal L(E) \cdot \langle P,P \rangle^{\cyc} \pmod{L^{\times}}; \]
\item[(c)] \[ \frac{d^2 L_p(\hf)(x_0,s)}{ds^2} \Big|_{s=1} = \mathcal L(E) \cdot \langle P,P \rangle^{\Sch} \pmod{L^{\times}}. \]
\end{enumerate}
\end{propo}
\begin{proof}
The first part follows from the main result of \cite{BD}, and the other two are \cite[Theorems D and E]{Ven}. We refer the reader to loc.\,cit. for a definition of the corresponding pairings.
\end{proof}

{\bf IV. Results beyond modular forms of weight 2.} The main result of Bertolini and Darmon \cite{BD} was generalized by Seveso \cite{Se1} to modular forms of even weight. Let us recall here his main result for the sake of completeness and to illustrate that most of our results generalize to the situation of weights $(k,1,1)$, by replacing the points over the elliptic curve by the corresponding Heegner cycles. Let $f_k \in S_k(N)$, where $N=pN^+N^-$ and $N^-$ is the squarefree product of an odd number of prime factors. The modular form corresponds, via the Jacquet-Langlands correspondence, to a modular form on a certain Shimura curve $X=X_{N^+,pN^-}$ uniformized by the $p$-adic upper half-plane. In this framework, Iovita and Spiess \cite{IS} constructed a Chow motive $\mathcal M_{k-2}$ attached to modular forms on $X$. Let $m=k/2-1$.

We fix $K/\mathbb Q$ a quadratic imaginary field extension, of discriminant $D_K$ prime to $pN$, such that $N^+$ is a product of primes that are split in $K$, while $pN^-$ is a product of primes that are inert in $K$; we further fix an order of $K$ of conductor $c$ prime to $ND_K$. Hence, one may consider a higher weight analogue of Heegner points, the Heegner cycles $y_{\psi}^{(n)} \in \CH^{m+1}(\mathcal M_n)$ attached to a character $\psi$. The $p$-adic \'etale Abel-Jacobi map takes the form \[ \AJ_p: \CH^{m+1}(\mathcal M_n) \rightarrow M_k^{\vee}. \]

In the following result, the Mazur--Kitagawa $p$-adic $L$-function is replaced by the $p$-adic $L$-function attached to the quadratic imaginary field and the character $\psi$, that we denote by $\mathcal L(\hf/K,\psi)(k,s)$ following the notations of \cite{Se1}.

\begin{propo}[Seveso]
The first derivative of $\mathcal L(f/K,\psi)(k,s)$ in the weight direction is given by \[ 2 \frac{d}{dx} \Big( \mathcal L(\hf/K,\psi)(x,x/2) \Big) \Big|_{x=k} = \AJ_p(y_{\psi}^{(n)})(f)+(-1)^m \AJ_p(y_{\bar \psi}^{(n)})(f). \]
\end{propo}

This suggests that some of our results can be transposed to a higher weight situation, replacing the points over the elliptic curve by the corresponding Heegner cycles. More precisely, the results relying on the work of Darmon, Lauder and Rotger on the elliptic Stark conjecture \cite{DLR1} can be adapted following the generalizations of Gatti and Guitart to higher weights \cite{GG}. Similarly, the construction of {\it derived} cohomology classes, anticipated in the introduction and developed in Section \ref{cohom-diag}, can be also carried out for general weights $(k,1,1)$.

\section{Derived diagonal cycles and an explicit reciprocity law}\label{cohom-diag}

\subsection{Diagonal cycles and an explicit reciprocity law}

Darmon and Rotger constructed in \cite{DR3} an element \[ \kappa(\hf,\hg,\hh) \in H^1(\mathbb Q, \mathbb V_{{\bf fgh}}^{\dag}) \] arising from the interpolation of diagonal cycles along the balanced region. An alternative construction has been given by Bertolini, Seveso and Venerucci \cite[Section 3]{BSV}. This class is symmetric in all three variables. Let \[ \res_p: H^1(\mathbb Q, \mathbb V_{{\bf fgh}}^{\dag}) \rightarrow H^1(\mathbb Q_p, \mathbb V_{{\bf fgh}}^{\dag}) \] denote the restriction map to the local cohomology at $p$, and set \[ \kappa_p(\hf,\hg,\hh) := \res_p(\kappa(\hf,\hg,\hh)) \in H^1(\mathbb Q_p, \mathbb V_{{\bf fgh}}^{\dag}). \]

One of the main results of both \cite{BSV} and \cite{DR3} is the proof of an explicit reciprocity law. As showed in loc.\,cit., the Galois representation $\mathbb V_{{\bf fgh}}^{\dag}$ is endowed with a four-step filtration
\begin{equation}\label{four-step}
0 \subset \mathbb V_{{\bf fgh}}^{++} \subset \mathbb V_{{\bf fgh}}^+ \subset \mathbb V_{{\bf fgh}}^- \subset \mathbb V_{{\bf fgh}}^{\dag}
\end{equation}
by $G_{\mathbb Q_p}$-stable $\Lambda_{{\bf fgh}}$-submodules of ranks 0, 1, 4, 7 and 8, respectively. Moreover, \[ \mathbb V_{{\bf fgh}}^+/\mathbb V_{{\bf fgh}}^{++} = \mathbb V_{\hf}^{\hg \hh} \oplus \mathbb V_{\hg}^{\hh \hf} \oplus \mathbb V_{\hh}^{\hf \hg}. \]

We discuss now the definition of $\mathbb V_{\hf}^{\hg \hh}$. Let $\Theta_{\hf}^{\hg \hh}$ be the $\Lambda_{{\bf fgh}}$-adic cyclotomic character whose specialization at a point of weight $(k,\ell,m)$ is $\varepsilon_{\cyc}^t$, with $t:=(-k+\ell+m)/2$, and let $\psi_{\hf}^{\hg \hh}$ be the unramified character of $G_{\mathbb Q_p}$ sending $\Fr_p$ to $\chi_f^{-1}(p) {\bf a}_p(\hf) {\bf a}_p(\hg)^{-1} {\bf a}_p(\hh)^{-1}$. Define $\mathbb U$ as the unramified $\Lambda_{{\bf fgh}}$-adic representation of $G_{\mathbb Q_p}$ given by several copies of the character $\psi_{\hf}^{\hg \hh}$, and let \[ \mathbb V_{\hf}^{\hg \hh}=\mathbb U(\Theta_{\hf}^{\hg \hh}). \] We finally introduce the $\Lambda$-adic Dieudonn\'e module \[ \mathbb D(\mathbb U) := (\mathbb U \hat \otimes \mathbb Z_p^{\nr})^{G_{\mathbb Q_p}}. \]

Then, one may construct a Perrin-Riou regulator map whose source is $H^1(\mathbb Q_p, \mathbb V_{\hf}^{\hg \hh}) \rightarrow \Lambda_{{\bf fgh}}$ and which interpolates either the Bloch--Kato logarithm or the dual exponential map, according to the value of a certain Hodge--Tate weight. In order to state their main properties, we need to introduce more terminology. Let $c=\frac{k+\ell+m-2}{2}$, and with the previous notations, define \[ \mathcal E^{\PR}(x,y,z) = \frac{1-p^{-c} \beta_{\hf_x} \alpha_{\hg_y} \alpha_{\hh_z}}{1-p^{-c} \alpha_{\hf_x} \beta_{\hg_y}\beta_{\hh_z}}. \]

The following result is discussed e.g. in \cite[Proposition 5.6]{DR3} and follows from the general theory of Perrin-Riou maps. For this statement we implicitly assume that neither $\beta_{\hf_x} \alpha_{\hg_y} \alpha_{\hh_z}$ nor $\alpha_{\hf_x} \beta_{\hg_y} \beta_{\hh_z}$ vanish. If this were the case, and as we will later see, we need to work with the expression \[ \mathcal E^{\PR}(x,y,z) = \frac{1-\chi_f(p) p^{\frac{k-\ell-m}{2}} \alpha_{\hf_x}^{-1} \alpha_{\hg_y} \alpha_{\hh_z}}{1-\bar \chi_f(p)p^{\frac{\ell+m-k-2}{2}} \alpha_{\hf_x} \alpha_{\hg_y}^{-1} \alpha_{\hh_z}^{-1}}, \] which agrees with the former in the non-exceptional case.

\begin{propo}\label{perrin}
There is a homomorphism (usually named {\it Perrin-Riou} regulator) \[ \mathcal L_{\hf,\hg \hh}: H^1(\mathbb Q_p, \mathbb V_{\hf}^{\hg \hh}) \rightarrow \mathbb D(\mathbb U) \] such that for all $\kappa_p \in H^1(\mathbb Q_p, \mathbb V_{\hf}^{\hg \hh})$ the image $\mathcal L_{\hf, \hg \hh}(\kappa_p)$ satisfies the following interpolation properties:
\begin{enumerate}
\item For all points $(x,y,z) \notin \mathcal W_{{\bf fgh}}^f$, \[ \nu_{x,y,z}(\mathcal L_{\hf,\hg \hh}(\kappa_p)) = \frac{(-1)^t}{t!} \mathcal E^{\PR}(x,y,z) \cdot \log_{\BK}(\nu_{x,y,z}(\kappa_p)), \]
\item For all points $(x,y,z) \in \mathcal W_{{\bf fgh}}^f$, \[ \nu_{x,y,z}(\mathcal L_{\hf,\hg \hh}(\kappa_p)) = (-1)^t \cdot (1-t)! \cdot \mathcal E^{\PR}(x,y,z) \cdot \exp_{\BK}^*(\nu_{x,y,z}(\kappa_p)). \]
\end{enumerate}
\end{propo}

Following \cite{DR3}, one can define
\begin{equation}\label{proj-quo}
\kappa_p(\hf,\hg,\hh)^f \in H^1(\mathbb Q_p, \mathbb V_{\hf}^{\hg \hh})
\end{equation}
as the projection of the local class $\kappa_p(\hf,\hg,\hh)$ to $\mathbb V_{\hf}^{\hg \hh}$. Let $\eta_{\hf^*}$, $\omega_{\hg^*}$ and $\omega_{\hh^*}$ be the canonical differentials attached to Hida families as introduced for instance in \cite[Section 10]{KLZ}. The following result has been independently established in \cite[Thoerem A]{BSV} and \cite[Theorem 10]{DR3}.

\begin{propo}\label{rec-law}
For any triplet of $\Lambda$-adic test vectors $(\tilde{\hf},\tilde{\hg},\tilde{\hh})$, the following equality holds in the ring of fractions of $\Lambda_{{\bf fgh}}$: \[ \langle \mathcal L_{\hf,\hg \hh}(\kappa_p(\hf,\hg,\hh)^f), \eta_{\tilde \hf^*} \otimes \omega_{\tilde \hg^*} \otimes \omega_{\tilde \hh^*} \rangle = \mathcal L_p^f(\tilde \hf, \tilde \hg, \tilde \hh). \]
\end{propo}

\begin{remark}
There exist analogue reciprocity laws for $\Lp^g$ and $\Lp^h$.
\end{remark}

We can also formulate an explicit reciprocity law for the improved $p$-adic $L$-function. Since along the region $k=\ell+m$ the Perrin-Riou map interpolates the dual exponential, we have that
\begin{equation}
\frac{1}{1-p^{-k+1} \alpha_{\hf_x} \beta_{\hg_y} \beta_{\hh_z}} \cdot \langle \exp_{\BK}^* (\kappa_p(\hf,\hg,\hh)^f(x,y,z)), \eta_{\tilde \hf_x^*} \otimes \omega_{\tilde \hg_y^*} \otimes \omega_{\tilde \hh_z^*} \rangle = \mathcal L_p^f(\tilde \hf_x, \tilde \hg_y, \tilde \hh_z)^*,
\end{equation}
and in particular the dual exponential map vanishes at $(x_0,y_0,y_0)$ (i.e. the class is crystalline) if and only if the improved $p$-adic $L$-function is zero at that point.

\begin{remark}
In \cite{GGMR}, the authors study the cohomology classes in a generic rank zero situation, where they are non-crystalline. This yields a formula for the special value $\Lp^g$ in terms of $\Lp^f$ in absence of exceptional zeros. Again, the key point is that each component of the cohomology class encodes information about a different $p$-adic $L$-function.
\end{remark}

\subsection{Vanishing of cohomology classes}

In \cite[Section 8.2]{BSV}, the authors deal with a situation where the numerator of the Perrin-Riou map $\mathcal L_{\hf,\hg \hh}$ vanishes, defining an improved map whose derivatives may be explicitly computed. We come back to this question later on. Let us analyze, firstly, the vanishing of the denominator of the Perrin-Riou map, but in the case of the Perrin-Riou map $\mathcal L_{\hg,\hh \hf}$, that is:
\begin{equation}
1-\bar \chi_g(p) p^{\frac{k-\ell+m-2}{2}}\alpha_{\hf_x}^{-1} \alpha_{\hg_y} \alpha_{\hh_z}^{-1} = 0.
\end{equation}
Since we have placed ourselves in the ordinary setting, a necessary condition for this to happen is $k+m=\ell+2$, which moreover suffices to guarantee the analyticity of the Euler factor in the denominator.

Hence, when $f$ is of weight 2 with split multiplicative reduction at $p$, and $g$ and $h$ are self-dual of the same weight ($h = g \otimes \chi_g^{-1}$), the denominator of the Perrin-Riou map vanishes. This means that we expect \[ \log_{\BK}(\kappa_p(\hf,\hg,\hh)^g(x,y,z))=0. \] However, the self-duality condition is not necessary for this vanishing, and following the conventions of the introduction in the case of weights $(2,1,1)$ (again, with $E$ of split multiplicative reduction), it suffices to impose that $\alpha_g \beta_h = 1$. This encompasses for example the case of theta series of quadratic fields where the prime $p$ is inert. Nevertheless, there are certain phenomena which are exclusive from the self-dual case: indeed, the fact that the Hida families interpolating both $g$ and $h$ keep the self-duality condition gives us a vanishing along the whole line $(2,\ell,\ell)$. We treat both the self-dual and the non self-dual case, emphasizing the main differences between them.

We begin by showing that when $\alpha_g \beta_h = 1$ and $g$ and $h$ are self-dual, the local class $\kappa_p(f,g_{\alpha},h_{\alpha})$ vanishes, using the techniques of our prior work \cite{RiRo1}. Although this is not strictly necessary since we will later see that the {\it whole} global class is zero, we believe that it may be instructive for the reader to compare the formalism of \cite{RiRo1}, which relies on the basic properties of the Perrin-Riou maps, with the more conceptual proof of \cite[Section 8]{BSV}, based on the geometric construction of an {\it improved} cohomology class.

\begin{propo}\label{longer}
With the running assumptions, the specialization of the $\Lambda$-adic cohomology class $\kappa_p(\hf,\hg,\hg^*)$ at $(x_0,y_0,y_0)$ vanishes, that is, $\kappa_p(f,g_{\alpha},g_{1/\beta}^*)=0$.
\end{propo}

\begin{proof}

We will follow the same strategy used in \cite[Theorem 3.5]{RiRo1}. First of all we show, invoking \cite[Theorem 7.1]{BSV}, that any specialization of the three-variable $\Lambda$-adic class at a point of weights $(2,\ell,\ell)$, with $\ell \geq 2$, is zero. In order to achieve this, we just use the comparison provided by the aforementioned result with the twisted class $\kappa^{\dag}$, twisting now in the $g$-variable, that is, applying the operator $\id \otimes w_p' \otimes \id$ according to the definitions given at the beginning of Section 7.2 of loc.\,cit., where $w_p$ stands for the Atkin--Lehner involution. As we later discuss, this class may be understood as an {\it improved} cohomology class, since it agrees with the former up to multiplication by the Euler factor \[ 1-\frac{\bar \chi(p) \alpha_{\hg_y}}{\alpha_{\hf_x} \alpha_{\hh_z}}p^{\frac{k-\ell+m-2}{2}}. \] This factor is zero over the line $(2,\ell,\ell)$ when we take Hida families such that $\hh=\hg^*$, since $\bar \chi(p) \alpha_{\hg_y} = \alpha_{\hh_y}$. Observe that we are implicitly using Lemma 8.4 of loc.\,cit., which assert that the class $\kappa(\hf,\hg,\hh)$ is symmetric in all three variables.

The second part of the proof consists on applying a limit argument componentwise, via the corresponding Perrin-Riou maps, to conclude that the limit when $\ell$ goes to one is also zero. For this last step, we look at the four different components of the local class $\kappa_p(f,g_{\alpha},g_{1/\beta}^*)$ corresponding to the {\it balanced} subspace $\mathbb V_{fgg^*}^+$. This suffices according to the results established in \cite[Corollary 7.2]{BSV} and following the notations of Section 6.2 in loc.\,cit., which asserts that the three-variable cohomology class lies in the balanced subspace. The components of the balanced subspace are denoted by $V_f^{gg^*}$, $V_g^{g^*f}$, $V_{g^*}^{fg}$ and $V_{fgg^*}^{++}$, where $V_f^{gg^*}$ stands for the specialization of $\mathbb V_{\hf}^{\hg \hg^*}$ and similarly for the other factors (recall the filtration of \eqref{four-step}).
\begin{itemize}
\item We first prove that the component associated to the rank one subspace $V_f^{gg^*}$ is zero. Observe that along the line $(2,\ell,\ell)$, the specialization of the module $H^1(\mathbb Q_p, \mathbb V_{\hf}^{\hg \hg^*})$ agrees with $H^1(\mathbb Q_p, \mathbb Z_p(\psi_{g_y}^{-2})(\ell-1))$, where $y$ is a point of weight $\ell$.
    Then, the Perrin-Riou map is an application
    \begin{equation}\label{liso}
    H^1(\mathbb Q_p, \Lambda_{\hg}(\psi_{\hg}^{-2}) \hat \otimes \Lambda(\underline{\varepsilon}_{\cyc})) \rightarrow \mathbb D(\Lambda_{\hg}(\psi_{\hg}^{-2})) \hat \otimes \Lambda.
     \end{equation}
     Since $\psi_{\hg}^{-2} \neq 1$, we have $H^0(\mathbb Q_p,\Lambda_{\hg}(\psi_{\hg}^{-2}))=0$ and it follows from \cite[Theorem 8.2.3]{KLZ} that the above map is an isomorphism. Moreover, using the same argument of the proof of the last step of \cite[Theorem 3.5]{RiRo1}, we conclude that the $\Lambda$-module of \eqref{liso} is non-canonically isomorphic to $\Lambda_{\hg}$. Therefore, and since infinitely many specializations vanish according to the previously quoted result of \cite{BSV}, the corresponding $H^1$ is zero.
\item The components associated to $V_{g}^{g^*f}$ and $V_{g^*}^{fg}$ are zero; this is because \[ H^1(\mathbb Q_p, \mathbb V_{\hg}^{\hg^* \hf}|_{(2,\ell,\ell)}) \simeq H^1(\mathbb Q_p, \Lambda_{\hg}(1)) \simeq \Lambda_{\hg} \oplus \Lambda_{\hg}, \] and although the Perrin-Riou map only kills one of the above two components, the restriction of the class is zero since again infinitely many specializations are zero.
\item For the remaining component, the one corresponding to $\mathbb V_{fgg^*}^{++}$, the same argument used in the first step works once we have established that the remaining projections vanish.
\end{itemize}
\end{proof}

Consider now the surface \[ \mathcal S = \mathcal S_{k,k+m-2,m} := \{(x,y,z) \in \mathcal W_{\hf} \times \mathcal W_{\hg} \times \mathcal W_{\hh} : \wt(x)+\wt(z)=\wt(y)+2 \}, \] and also the line \[ \mathcal C := \{(x,y,z) \in \mathcal W_{\hf} \times \mathcal W_{\hg} \times \mathcal W_{\hh} : \wt(x)=2, \quad \wt(y)=\wt(z) \}. \] Observe that the surface $\mathcal S$ is just a finite cover of the plane in $\mathcal W^3$ arising as the Zariski closure of weights $(k,k+m-2,m)$.

Using the results of \cite[Section 8.2]{BSV}, we may upgrade Proposition \ref{longer} to the vanishing of the global class $\kappa(f,g_{\alpha},h_{\alpha})$ when $\alpha_g \beta_h = 1$ (and hence we can work beyond the setting of the adjoint, covering for example the case of theta series of quadratic fields where the prime $p$ remains inert).

In particular, we have the following result.
\begin{propo}\label{tot-zero}
The global class $\kappa(\hf,\hg,\hh)$ vanishes along the line $\mathcal C$ in the self dual case. Moreover, the class $\kappa(f,g_{\alpha},h_{1/\beta})$ is zero when $\alpha_g \beta_h = 1$.
\end{propo}
\begin{proof}
Following again \cite[Section 8.2]{BSV}, there is an {\it improved} class $\kappa_g^*(\hf,\hg,\hh)$ along the surface $\mathcal S$ satisfying
\begin{equation}
\kappa(\hf,\hg,\hh)_{|\mathcal S} = \Big( 1-\frac{\bar \chi(p) \alpha_{\hg_y}}{\alpha_{\hf_x} \alpha_{\hh_z}} \Big) \kappa_g^*(\hf,\hg,\hh).
\end{equation}
Hence, the vanishing of $\kappa(f,g_{\alpha},h_{1/\beta})$ follows from the vanishing of the corresponding Euler factor.
\end{proof}

\subsection{Derived classes and reciprocity laws}

Following the analogy with \cite{RiRo1}, let us focus firstly on the self-dual case to discuss the notion of {\it derived} classes. We shrink the weight space $\mathcal W$ to a rigid-analytic open disk $\mathcal U \subset \mathcal W$ centered at 2 at which the finite cover $w: \mathcal W_{\hf} \rightarrow \mathcal W$ restricts to an isomorphism $w: \mathcal U_{\hf} \xrightarrow{\sim} \mathcal U$ with $x_0 \in \mathcal U_{\hf}$.
Let $\Lambda_{\mathcal U_{\hf}} = \mathcal O(\mathcal U_{\hf})$ denote the Iwasawa algebra of analytic functions on $\mathcal U_{\hf}$ whose supremum norm is bounded by $1$. Shrink likewise $\mathcal C$ and $\mathcal S$ so that projection to the weight space restricts to an isomorphism with $\mathcal U$ and $\mathcal U \times \mathcal U$ respectively. Having done that, their associated Iwasawa algebras are respectively $\mathcal O(\mathcal C) = \Lambda_{\mathcal U_{\hf}} \simeq \mathbb Z_p[[X]]$ and $\mathcal O(\mathcal S) = \Lambda_{\mathcal U_{\hf}} \hat \otimes \Lambda_{\mathcal U_{\hh}} \simeq \mathbb Z_p[[X,Z]]$. The isomorphism $\Lambda_{\mathcal U_{\hf}} \simeq \mathbb Z_p[[X]]$ is not canonical and depends on the choice of an element $\gamma \in \Lambda_{\mathcal U_{\hf}}^{\times}$ which is sent to $1+X$.


Then, consider the short exact sequence of $\mathbb Z_p$-modules \[ 0 \rightarrow \mathbb Z_p[[X,Z]] \xrightarrow{\cdot X} \mathbb \mathbb Z_p[[X,Z]] \rightarrow \mathbb Z_p[[Z]] \rightarrow 0. \] Under the usual isomorphisms, $\Lambda_{\hf}$ may be identified with $\mathbb Z_p[[X]]$ after fixing a topological generator $\gamma$ of $\Lambda_{\mathcal U_{\hf}}^{\times}$ and sending $[\gamma]$ to $1+X$. Then, $\Lambda_{\hf} \hat \otimes \Lambda_{\hh}$ becomes isomorphic to $\mathbb Z_p[[X,Z]]$ and the previous exact sequence may be recast as
\begin{equation}\label{ses}
0 \rightarrow \mathcal O_{\mathcal S} \xrightarrow{\delta} \mathcal O_{\mathcal S} \rightarrow \mathcal O_{\mathcal C} \rightarrow 0
\end{equation}
with $\delta = (\gamma-1) \otimes 1$ in $ \mathcal O_{\mathcal S} \simeq \Lambda_{\hf} \hat\otimes \Lambda_{\hh}$.

\begin{propo}\label{deri}
In the self-dual case, there exists a unique derived class $\kappa_{\gamma}'(\hf,\hg,\hg^*) \in H^1(\mathbb Q, \mathbb V_{{\bf fgg^*}|\mathcal S})$ such that \[ \kappa(\hf,\hg,\hg^*)|_{\mathcal S} = \delta \cdot \kappa_{\gamma}'(\hf,\hg,\hg^*). \]
\end{propo}
\begin{proof}
This follows by considering the long exact sequence in cohomology attached to \eqref{ses}: \[ H^0(\mathbb Q, \mathbb V_{\hf \hg \hg^*|\mathcal C}) \rightarrow H^1(\mathbb Q, \mathbb V_{\hf \hg \hg^*|\mathcal S}) \rightarrow H^1(\mathbb Q, \mathbb V_{\hf \hg \hg^*|\mathcal S}) \rightarrow H^1(\mathbb Q, \mathbb V_{\hf \hg \hg^*|\mathcal C}). \] Since the restriction of $\kappa(\hf,\hg,\hg^*)$ to $H^1(\mathbb Q, \mathbb V_{\hf \hg \hg^*|\mathcal C})$ is zero by Proposition \ref{tot-zero}, one may assure the existence of a {\it derived} class as in the statement, which is moreover unique due to the vanishing of the $H^0$ for weight reasons (the Hodge--Tate weights corresponding to the balanced part cannot be zero, as shown in \cite[Corollary 5.3]{DR3}).
\end{proof}

Normalizing by $\log_p(\gamma)$, the specializations of this class over the line $(2,\ell,\ell)$ can be proved to be independent of the choice of $\gamma$.

In general, if we are no longer in the self-dual case but the condition $\alpha_g \beta_h = 1$ still holds, the notion of derived class makes sense at the point $(x_0,y_0,z_0)$. For that purpose, let $\mathcal D$ stand for the codimension two subvariety \[ \mathcal D := \{ (x,y,z) \in \mathcal W_{\hf} \times \mathcal W_{\hg} \times \mathcal W_{\hh} : \wt(x) = \wt(y) + 1, \quad z = z_0 \}. \]
The following result is the analogue of \cite[Proposition 3.13]{RiRo1} and its proof follows from the same argument of Proposition \ref{deri}.

\begin{propo}
Assume that $\alpha_g \beta_h = 1$. Then, there exists a unique class $\kappa_{\gamma}'(\hf,\hg,\hh) \in H^1(\mathbb Q, \mathbb V_{{\bf fgh}|\mathcal D})$ such that \[ \kappa(\hf,\hg,\hh)|_{\mathcal D} = \delta \cdot \kappa_{\gamma}'(\hf,\hg,\hh). \]
\end{propo}

Let $\mathcal L = \frac{\alpha_g'}{\alpha_g} -\frac{\alpha_f'}{\alpha_f}$, and consider the normalization of $\kappa_{\gamma}'(\hf,\hg,\hh)$ by $\gamma$, that is, \[ \kappa'(\hf,\hg,\hh) = \frac{\kappa_{\gamma}'(\hf,\hg,\hh)}{\log_p(\gamma)}. \]

\begin{theorem}\label{reclaw}
The logarithm of the derived class satisfies the following \[ \langle \log_{\BK}(\kappa_p'(\hf,\hg,\hh)^g(x_0,y_0,z_0)), \eta_f \otimes \omega_g \otimes \omega_h \rangle = \mathcal L \cdot \Lp^{g_{\alpha}}(\hf,\hg,\hh)(x_0,y_0,z_0) \pmod{L^{\times}}, \] where as in \eqref{proj-quo} the superindex $g$ refers to the projection to $\mathbb V_{\hg}^{\hh \hf}$.
\end{theorem}
\begin{proof}
Consider the reciprocity law of Proposition \ref{rec-law}, now for $\Lp^{g_{\alpha}}(\hf,\hg,\hh)$, restricted to $\mathcal D$, and multiply both sides by the Euler factor in the denominator of the Perrin-Riou map. Then, we have an equality of the form \[  \Big( 1 - \frac{\bar \chi(p) \alpha_{\hg_y}}{p \alpha_{\hf_x} \alpha_{\hh_z}} \Big) \cdot \langle \log_{\BK}(\kappa_p(\hf,\hg,\hh)^g), \eta_{\hf} \otimes \omega_{\hg} \otimes \omega_{\hh} \rangle = \Big( 1 - \frac{\alpha_{\hf_x} \alpha_{\hh_z}}{\bar \chi(p) \alpha_{\hg_y}} \Big) \cdot \Lp^{g_{\alpha}}(\hf,\hg,\hh),  \] since along $\mathcal D$ the Perrin-Riou interpolates the Bloch--Kato logarithm. At the point $(x_0,y_0,z_0)$ both the cohomology class at the left hand side and the Euler factor at the right are zero. Taking derivatives along the direction $(k+1,k,1)$, and evaluating then at the point $(x_0,y_0,z_0)$, the result follows (see \cite[Remark 4.8]{R1} for a more exhaustive discussion on the identifications we are considering).
\end{proof}

An analogue formula holds for any point over the line $(2,\ell,\ell)$ in the self-dual case, but of course the description of the $\mathcal L$-invariant is not so explicit and relies on the results of \cite{Se1}.

It may be instructive to compare this {\it derived} cohomology class with the {\it improved} cohomology class considered by Bertolini, Seveso and Venerucci. We can prove the following.

\begin{propo}
Consider the map given by the projection \[ \phi_g^{hf}: H^1(\mathbb Q, \mathbb V_{\hf \hg \hh|\mathcal S}) \rightarrow H^1(\mathbb Q_p,\mathbb V_{\hg}^{\hh \hf}|_{\mathcal S}). \]
Then, there is a relation between the improved class $\phi_g^{hf}(\kappa_g^*(\hf,\hg,\hh))$ and $\phi_g^{hf}(\kappa'(\hf,\hg,\hh))$, given by \[ \phi_g^{hf}(\kappa'(\hf,\hg,\hh)) = \mathcal L \cdot \phi_g^{hf}(\kappa_g^*(\hf,\hg,\hh)) \pmod{L^{\times}}. \]
\end{propo}
\begin{proof}
This is proved by applying the map $\langle \log_{\BK}(\cdot),\eta_{\hf} \otimes \omega_{\hg} \otimes \omega_{\hh} \rangle$ to both sides, and then comparing the results. For that purpose, we use that the Euler factors involved in the Perrin-Riou map are analytic along $\mathcal S$ and can be cancelled out. That way, we obtain an {\it improved} reciprocity law \[ \Lp^{g_{\alpha}}(\hf,\hg,\hh)(x,y,z) = \langle \log_{\BK}(\phi_g^{hf}(\kappa_g^*(\hf,\hg,\hh)^g(x,y,z))), \eta_{\hf_x} \otimes \omega_{\hg_y} \otimes \omega_{\hh_z} \rangle \pmod{L^{\times}}, \] which holds for all the points $(x,y,z)$ of $\mathcal S$.
\end{proof}

Finally, we point out that we may expect a relation between $\kappa_p'(\hf,\hg,\hh)(x,y,z)$ and the Gross--Kudla--Schoen cycle of \cite{DR1}, that we denote by $\Delta_{k,\ell,m} \in H^1(\mathbb Q, V_{fgh}((4-k-\ell-m)/2)$. In particular, we expect the following result to be true (or at least, a slight variant of it). Here, $\loc_p$ stands for the localization at $p$-map.

\begin{question}
Can we establish that, up to multiplication by a non-zero constant in $L^{\times}$ and for any point $(x,y,z)$ of weights $(2,\ell,\ell)$ with $\ell \geq 2$, we have the equality \[ \kappa_p'(\hf,\hg,\hh)(x,y,z) = \mathcal L \cdot \loc_p(\Delta_{2,\ell,\ell})? \]
\end{question}

Of course, this would require the proof of an analogue result to \cite[Theorem 5.1]{DR1} in a situation where $f$ has split multiplicative reduction.

\section{Derivatives of triple product p-adic L-functions}

In this section, we discuss a variant of the elliptic Stark conjecture for the derivative of the triple product $p$-adic $L$-function $\Lp^f$ in a situation of exceptional zeros. As before, we keep the assumption that $f$ has split multiplicative reduction at $p$ and that an exceptional zero condition occurs.

There are two main instances we want to consider: the rank zero situation and the rank two situation. While the former is quite well understood after the results developed in \cite{BSV} and \cite{BSV2}, the latter is more subtle and we will propose a conjectural formula in this scenario. Along this section, by the word {\it rank}, we refer to the rank of the $V_{gh}$-isotypic component of $E(H)$. According to our general assumptions on the local signs, the rank is always even. The $V_{gh}$-component of $E(H)$ is endowed with an inclusion in the Selmer group, that is, \[ \Hom_{G_{\mathbb Q}}(E(H), V_{gh}) \simeq (E(H) \otimes V_{gh}^{\vee})^{G_{\mathbb Q}} \subset H_{\fin}^1(\mathbb Q, V_{fgh}), \] where $H_{\fin}^1(\mathbb Q,V_{fgh})$ is the group of extensions of $\mathbb Q_p$ by $V_{fgh}$ in the category of $\mathbb Q_p$-linear representations of $G_{\mathbb Q}$ which are crystalline at $p$.

Recall that for higher ranks the computations performed in \cite{DLR1} lead us to expect that the special value $\Lp^{g_{\alpha}}$ presented in the introduction is zero, and that the second derivative of $\Lp^f$ along the $f$-direction vanishes, too. The odd rank situation is equally interesting, and we hope to come back to this question in a further work. We keep the notations of the previous section.

\begin{propo}
The value $\Lp^f(\hf,\hg,\hh)(x_0,y_0,z_0)$ is zero. Moreover, the jacobian matrix of $\Lp^f(\hf,\hg,\hh)$ at the point $(x_0,y_0,z_0)$ is given by \[ (0 \quad \mathcal L_{g_{\alpha}}-\mathcal L_f \quad \mathcal L_{h_{\alpha}}-\mathcal L_f) \cdot \Lp^f(\hf,\hg,\hh)^*. \]
\end{propo}
\begin{proof}
This directly follows from \cite[Proposition 8.2]{BSV}.
\end{proof}

\begin{remark}
Observe that, towards the rationality conjectures we are interested in, the value $\Lp^f(\hf,\hg,\hh)^*$ is an algebraic number, and it is non-zero if and only if the cohomology class $\kappa(f,g_{\alpha},h_{\alpha})$ is non-crystalline.
\end{remark}

In particular, the derivative along the direction $(2+k,1,1)$ vanishes, and along the direction $(2,1+\ell,1+\ell)$ is given by $\mathcal L_{g_{\alpha}}+\mathcal L_{h_{\alpha}}-2\mathcal L_f$, up to an explicit algebraic number in the number field $L$.

Suppose from now on that $\Lp^f(\hf,\hg,\hh)^*$ vanishes at $(x_0,y_0,z_0)$. Therefore, the cohomology class $\kappa(f,g_{\alpha},h_{\alpha})$ is crystalline, and following \cite[Section 2.1]{BSV2} we may define a {\it new} Bloch--Kato logarithm, denoted by $\log_{\beta \beta}$ in loc.\,cit. Roughly speaking, it can be understood as a projection to the rank one subspace $V_{fgh}^{++}$ arising in the filtration \eqref{four-step}, followed by the Bloch--Kato logarithm and the pairing with $\omega_f \otimes \omega_g \otimes \omega_h$. To be coherent with the other notations we will need later on, write $\log^{++}$ for this map. Alternatively, we may consider the local class $\kappa_p(f,g_{\alpha},h_{\alpha})$ and take its decomposition according to the action of the Frobenius element, in such a way that $\kappa_{\beta \beta}$ is the part corresponding to the $(\beta_g,\beta_h)$ component.

Assume further that $\alpha_g \alpha_h = 1$ (in particular, this also implies that $\beta_g \beta_h = 1$). The following result is the content of \cite[Section 2.1]{BSV2}.

\begin{propo}\label{beta-beta}
Under the given conditions, the value $\Lp^f(\hf,g_{\alpha},h_{\alpha})$ vanishes and \[ \frac{d^2}{dx^2} \Lp^f(\hf,g_{\alpha},h_{\alpha})|_{x=x_0} = \frac{1}{2\ord_p(q_E)} \cdot (1-p^{-1})^{-1} \cdot \log^{++}(\kappa_p(f,g_{\alpha},h_{\alpha})). \]
\end{propo}
\begin{remark}
In the adjoint case, when we take $h_{1/\alpha}=g_{1/\beta}^*$ we do have a relation between $\mathcal L_g$ and $\mathcal L_h$: indeed \[ \frac{(1/\alpha_g)'}{1/\alpha_g} = -\frac{\alpha_g'}{\alpha_g}; \] however when $h_{1/\alpha}=g_{1/\alpha}^*$ both quantities are a priori unrelated.
\end{remark}

\subsection{A conjecture for the second derivative}

As we have discussed before, the improved $p$-adic $L$-function $\Lp^f(\hf,\hg,\hh)^*$ interpolates an explicit non-zero multiple of $L(f \otimes g \otimes h,1)$, and we expect this value to be zero when the rank of the corresponding isotypic component of the Selmer group is two. In those cases, we would like to compare the Kato class with a basis of $(E(H) \otimes V_{gh}^{\vee})^{G_{\mathbb Q}}$, that we write as $\{P,Q\}$. We also assume that $H_{\fin}^1(\mathbb Q, V_{fgh})$ has dimension 2.

To fix notations, observe that $V_{gh}$ decomposes as a $G_{\mathbb Q_p}$-module as the direct sum of four different lines $V_{gh}^{\alpha \alpha}:= V_g^{\alpha_g} \otimes V_h^{\alpha_h},\ldots,V_{gh}^{\beta \beta}$. After choosing a basis of $V_{gh}^{\vee}$, we may write this decomposition as \[ V_{gh} = L \cdot e_{\alpha \alpha}^{\vee} \oplus L \cdot e_{\alpha \beta}^{\vee} \oplus L \cdot e_{\beta \alpha}^{\vee} \oplus L \cdot e_{\beta \beta}^{\vee}, \] where \[ \Fr_p(e_{\lambda \mu}^{\vee}) = a_{\lambda \mu} \cdot e_{\lambda \mu}^{\vee} \quad \text{ for any } \quad \lambda,\mu \in \{ \alpha,\beta \}. \] Here, $a_{\lambda \mu} = \beta_g \beta_h$ if $(\lambda,\mu)=(\alpha,\alpha)$ and similarly in the other three cases.

In particular, restricting the elements $\{P,Q\}$ to a decomposition group at $p$ gives expressions \[ P = P_{\beta \beta} \otimes e_{\beta \beta}^{\vee} + P_{\beta \alpha} \otimes e_{\beta \alpha}^{\vee} + P_{\alpha \beta} \otimes e_{\alpha \beta}^{\vee} + P_{\alpha \alpha} \otimes e_{\alpha \alpha}^{\vee}, \] and similarly for $Q$, where as recalled in the introduction $\Frob_p$ acts on $P_{\beta \beta}$ with eigenvalue $\beta_g \beta_h$ and analogously for the remaining components.

\begin{conj}\label{conjFa}
Under the running assumptions, the following equality holds: \[ \frac{d^2}{dx^2} \Lp^f(\hf,g_{\alpha},h_{\alpha})|_{x=x_0} = \log_p(P_{\beta \beta}) \log_p(Q_{\alpha \alpha}) - \log_p(Q_{\beta \beta}) \log_p(P_{\alpha \alpha}) \pmod{L^{\times}}. \]
\end{conj}

This conjecture can be seen as a quite natural analogue for the first part of Proposition \ref{3ven}; that is, we are proposing an expression for the second derivative of the $p$-adic $L$-function along the line $(k,k/2)$ since the central critical point corresponding to $(k,1,1)$ is precisely $k/2$. It would be interesting to understand the derivatives along different directions; we expect that they would be related with appropriate height pairings. See \cite{CH} for an approximation to that question when $g$ and $h$ are theta series attached to the same quadratic imaginary field where the prime $p$ splits.

\subsection{Some reducible cases}

We continue by recalling some factorization formulas in special cases where the representation $V_{gh}$ becomes reducible. See \cite[Section 2]{DLR2} for a complete discussion of the different cases where this may occur.

A first case occurs when $g$ and $h$ are theta series of the same quadratic field $K$, but the behavior is ostensibly different depending on whether $K$ is real or imaginary, and on whether $p$ is inert or split in $K$. While the inert case was worked out in \cite{BSV2}, the split case was not considered in loc.\,cit. However, it turns out that it is not specially interesting, at least when $K$ is imaginary: the second derivative of $\Lp^f$ along the $x$-direction is 0 for trivial reasons.

\begin{remark}
It may be tempting to prove a factorization formula for $\Lp^f$ as in \cite{CR}, or even when all three variables $(k,\ell,m)$ are allowed to move along a Hida family. However, the two-variable Castella's $p$-adic $L$-functions considered in loc.\,cit. would have infinity types \[ \Big( \frac{k+\ell+m}{2}-1, \frac{k+\ell+m}{2}-\ell-m+1 \Big), \quad \Big( \frac{k+\ell+m}{2}-m, \frac{k+\ell+m}{2}-\ell \Big). \] This precludes the possibility of comparing the different $p$-adic $L$-values along the region of classical interpolation, since they are disjoint.
\end{remark}

Finally, in the case where $h=g^*$, the situation is also quite simple and the right hand of the conjecture is zero when the component corresponding to the adjoint has rank two. For details on that, see the case by case analysis, completely analogue to our situation, of \cite{DR2.5}.

\begin{propo}\label{redu}
Conjecture \ref{conjFa} holds whenever (a) $g$ and $h$ are theta series of an imaginary quadratic field where $p$ splits, $V_{gh} = V_{\psi_1} \oplus V_{\psi_2}$, with each component of rank one; (b) $g$ and $h$ are theta series of a quadratic field where $p$ is inert, $V_{gh} = V_{\psi_1} \oplus V_{\psi_2}$, and either $\psi_1$ or $\psi_2$ is a genus character.
\end{propo}
\begin{proof}
Consider first the case of imaginary quadratic fields, where we can prove that both the left and the right hand side are zero for trivial reasons. For that purpose, recall the notations introduced in the discussion before Proposition \ref{beta-beta}. In order to see that the second derivative vanishes, it is enough to conclude that the component $\kappa_{\beta \beta} = 0$, and this follows after adapting the results of \cite[Section 4.3]{DR2.5} to the multiplicative situation, where one may invoke the discussion of \cite{CR}. In particular, if we assume without loss of generality that $P_{\beta \alpha} \neq 0$, then $P_{\alpha \alpha} = P_{\beta \beta} = 0$, and similarly $Q_{\alpha \beta} = Q_{\beta \alpha} = 0$. See \cite[Section 4]{GGMR} for a similar treatment of an analogue situation.

The case of theta series for quadratic fields where the prime is inert follows from the main results of Bertolini--Seveso--Venerucci \cite[Section 3]{BSV2}, taking into account the identifications among the different eigenspaces for the Frobenius action of e.g.\,\cite[Section 3.3]{DLR1} and \cite[Sections 4.3, 4.4]{DR2.5}.
\end{proof}

\subsection{The conjecture in other settings}

We would like to make some comments regarding the case $\alpha_g \beta_h = 1$. Observe that the previous Euler factor that gave rise to the improved $p$-adic $L$-function does not vanish, but the factors \[ 1-\frac{\chi_f(p) \alpha_{\hg_y} \beta_{\hh_z}}{\alpha_{\hf_x}}p^{\frac{k-\ell-m}{2}}  \quad \text{ and } \quad 1-\frac{\chi_f(p) \beta_{\hg_y} \alpha_{\hh_z}}{\alpha_{\hf_x}}p^{\frac{k-\ell-m}{2}} \] do. The first one is analytic along the region $\mathcal S_{k+m=\ell+2}$, while the second is analytic along the region $\mathcal S_{k+\ell=m+2}$. In this case we cannot assure the existence of an improved $p$-adic $L$-function, but at least we can guarantee that $\Lp^f(\hf,\hg,\hh)$ vanishes.

We get indeed a very similar result.
\begin{propo}
The value $\Lp^f(\hf,\hg,\hh)(x_0,y_0,z_0)=0$. Moreover, the jacobian matrix of $\Lp^f(\hf,\hg,\hh)$ at $(x_0,y_0,z_0)$ is given by an $L$-multiple of \[ (0 \quad \mathcal L_{g_{\alpha}}-\mathcal L_f \quad \mathcal L_{h_{\alpha}}-\mathcal L_f). \]
\end{propo}

Observe that Conjecture \ref{conjFa} still makes sense in this framework. And again, we can also take the derivative along the line $(2+k,1+k,1)$ and we would expect to relate it with an explicit multiple of an appropriate height pairing $\langle P,P \rangle$.

\section{Applications to the elliptic Stark conjecture}\label{conj-S2}

\subsection{Interplay between both settings and a conjecture of Darmon--Rotger}

Let $H$ denote the smallest number field cut out by the representation $V_{gh}$, with coefficients in a finite extension $L/\mathbb Q$. By enlarging it if necessary, assume throughout that $L$ contains both the Fourier coefficients of $g$ and $h$, and the roots of their $p$-th Hecke polynomials. Fix a prime ideal $\wp$ of $H$ lying above $p$, thus determining an embedding $H \subset H_p \subset \bar{\mathbb Q}_p$ of $H$ into its completion $H_p$ at $\wp$, and an arithmetic Frobenius $\Fr_p \in \Gal(H_p/\mathbb Q_p)$. Due to our regularity assumptions, $V_g$ and $V_h$ decompose as \[ V_g = V_g^{\alpha} \oplus V_g^{\beta}, \quad V_h = V_h^{\alpha} \oplus V_h^{\beta}, \] where $\Fr_p$ acts on $V_g^{\alpha}$ with eigenvalue $\alpha_g$, and similarly for the remaining summands.

Fix eigenbases $\{e_g^{\alpha},e_g^{\beta}\}$ and $\{e_h^{\alpha},e_h^{\beta}\}$ of $V_g$ and $V_h$, respectively, which are compatible with the choice of the basis for $V_{gh}$, i.e., \[ e_{\alpha \alpha} = e_g^{\alpha} \otimes e_h^{\alpha}, \quad e_{\alpha \beta} = e_g^{\alpha} \otimes e_h^{\beta}, \quad e_{\beta \alpha} = e_g^{\beta} \otimes e_h^{\alpha}, \quad e_{\beta \beta} = e_g^{\beta} \otimes e_h^{\beta} \] (recall that in previous sections we were using the dual basis). Let $\eta_{g_{\alpha}} \in (H_p \otimes V_g^{\beta})^{G_{\mathbb Q_p}}$ and $\omega_{h_{\alpha}} \in (H_p \otimes V_g^{\alpha})^{G_{\mathbb Q_p}}$ denote the canonical periods arising as the weight one specializations of the $\Lambda$-adic periods $\eta_{\hg}$ and $\omega_{\hh}$ introduced in \cite[Section 10.1]{KLZ}.
Then, we can define $p$-adic periods $\Xi_{g_{\alpha}} \in H_p^{\Fr_p=\beta_g^{-1}}$ and $\Omega_{h_{\alpha}} \in H_p^{\Fr_p=\alpha_h^{-1}}$ by setting \[ \Xi_{g_{\alpha}} \otimes e_g^{\beta} = \eta_{g_{\alpha}}, \quad \Omega_{h_{\alpha}} \otimes e_h^{\alpha} = \omega_{h_{\alpha}}, \]
and
\begin{equation}
\mathcal L_{g_{\alpha}} := \frac{\Omega_{g_{\alpha}}}{\Xi_{g_{\alpha}}} \in (H_p)^{\Fr_p=\frac{\beta_g}{\alpha_g}}.
\end{equation}

At the same time, recall that $u_{g_{\alpha}}$ is the Stark unit attached to the adjoint representation of $g_{\alpha}$, which arises as a normalization term in the conjectures of \cite{DLR1} and \cite{DLR2} involving a {\it second-order} regulator. Then, it was conjectured by Darmon and Rotger \cite{DR2.5} that
\begin{equation}
\mathcal L_{g_{\alpha}} = \log_p(u_{g_{\alpha}}) \pmod{L^{\times}}.
\end{equation}
This relation gives a relatively easy interpretation of the apparently mysterious unit $u_{g_{\alpha}}$. This suggests that more natural descriptions of this object should be available, involving only the arithmetic of the modular form $g$. However, this conjecture seems to be hard to prove, even in cases where the elliptic Stark conjecture is known (theta series of quadratic imaginary fields where the prime $p$ splits). The main difficulty is the lack of an explicit description of the periods $\Omega_{g_{\alpha}}$ and $\Xi_{g_{\alpha}}$: in weights greater than one, these periods can be understood as certain algebraic numbers and be explicitly described, but in weight one this description is no longer available and $\Omega_{g_{\alpha}}$ and $\Xi_{g_{\alpha}}$ are $p$-adic transcendental numbers.

The main point of this section is that the knowledge of different conjectures involving these periods can be enough to determine the value of the ratio $\mathcal L_{g_{\alpha}}$. Indeed, the generalized cohomology classes described in Section \ref{cohom-diag} can be decomposed as the sum of different components, each one encoding information about different $p$-adic $L$-functions. When combining these results, we may relate the different periods which are involved.

As a first application of this technique, let us prove a result of this kind using the theory of Beilinson--Flach elements. This corresponds to the limit case where the modular form $f$ is Eisenstein and the arithmetic governing the triple product are ostensibly different. For the following discussion, the notations are the same of \cite{RiRo2}. Let $U_{gg^*} = \mathcal O_H^{\times} \otimes L$ and $U_{gg^*}[1/p] = \mathcal O_H[1/p]^{\times} \otimes L$, and assume that the hypothesis (H1)-(H3) of the introduction of \cite{RiRo2} hold. Fix a basis $\{u,v\}$ of the two dimensional space $(U_{gg^*}[1/p]/p^{\mathbb Z} \otimes \ad^0(V_g))^{G_{\mathbb Q}}$ such that $u \in (\mathcal O_H^{\times} \otimes \ad^0(V_g))^{G_{\mathbb Q}}$. As in the case of elliptic curves, these unit groups are endowed with a Frobenius action, since the restriction to a decomposition group allows us to decompose $\ad^0(V_g) = L^1 \oplus L^{\alpha/\beta} \oplus L^{\beta/\alpha}$ and we may take the projection of $u$ and $v$ to each of those components. Let
\[ \begin{array}{ccc}
R_{g_{\alpha}} &=& \log_p(u_1) \log_p(v_{\alpha/\beta}) - \log_p(v_1) \log_p(u_{\alpha/\beta}), \\ R_{g_{\beta}} &=& \log_p(u_1) \log_p(v_{\beta/\alpha}) - \log_p(v_1) \log_p(u_{\beta/\alpha})
\end{array} \]
be the regulators which appear in the formulation of the main conjecture of \cite{DLR2} and \cite{RiRo2}.

\begin{propo}\label{evidence}
Assume that $R_{g_{\alpha}}$ and $R_{g_{\beta}}$ are both non-zero. Then, \[ \frac{\mathcal L_{g_{\alpha}}}{\mathcal L_{g_{\beta}}} = \frac{\log_p(u_{g_{\alpha}})}{\log_p(u_{g_{\beta}})} \pmod{L^{\times}}. \]
\end{propo}
\begin{proof}
Recall the maps $\log^{+-}$ and $\log^{-+}$ introduced in \cite[Section 3.3]{RiRo2} as the composition of the corresponding projection maps from $V_{gh}$, the Bloch--Kato logarithm, and the pairing with the canonical differentials. Apply then \cite[Proposition 4.3]{RiRo2} twice, first with the map $\log^{-+}$ (and hence taking the $\beta/\alpha$ component of both $u$ and $v$), and then with the map $\log^{+-}$ (taking the $\alpha/\beta$ component of both $u$ and $v$). Then, comparing both expressions we have that \[ \Xi_{g_{\alpha}} \cdot \Omega_{g_{1/\alpha}^*} \cdot \log_p(u_{g_{\alpha}}) = \Omega_{g_{\alpha}} \cdot \Xi_{g_{1/\alpha}^*} \cdot \log_p(u_{g_{\beta}}) \pmod{L^{\times}}. \] We now proceed as in \cite[Section 5.2]{RiRo1} (see the discussion after display (56)), observing that \[ \Omega_{g_{1/\alpha}^*} = \Xi_{g_{\beta}}^{-1}, \quad \Xi_{g_{1/\alpha}^*} = \Omega_{g_{\beta}}^{-1} \pmod{L^{\times}}, \] and we are done.
\end{proof}

We would like to go a step beyond and aim for stronger results, so in a certain way we would like to keep the period attached to $h$ fixed and vary just the one attached to $g$, which would yield the desired equality.

We do this by analyzing first the prototypical case of the elliptic Stark conjecture, where the Selmer group is two-dimensional and we may fix a basis $\{P,Q\}$ of the $L$-vector space \[ (E(H) \otimes V_{gh}^{\vee})^{G_{\mathbb Q}}. \] For the following Proposition we assume the hypothesis discussed in the introduction of \cite{DLR1}, and in particular, that $L(f \otimes g \otimes h,1)=0$. Recall the decomposition \[ P = P_{\beta \beta} \otimes e_{\beta \beta}^{\vee} + P_{\beta \alpha} \otimes e_{\beta \alpha}^{\vee} + P_{\alpha \beta} \otimes e_{\alpha \beta}^{\vee} + P_{\alpha \alpha} \otimes e_{\alpha \alpha}^{\vee}, \] and similarly for $Q$.

Define the regulators \[ \Reg_{g_{\alpha}}(V_{gh}) = \log_p(P_{\beta \beta}) \log_p(Q_{\beta \alpha})-\log_p(Q_{\beta \beta}) \log_p(P_{\beta \alpha}) \] and \[ \Reg_f(V_{gh}) = \log_p(P_{\beta \beta}) \log_p(Q_{\alpha \alpha})-\log_p(Q_{\beta \beta}) \log_p(P_{\alpha \alpha}). \]


To shorten notations, write \[ \log^{-+}(\kappa) = \langle \log_{\BK}(\kappa_p^g), \omega_f \otimes \eta_g \otimes \omega_h \rangle, \] and whenever $\kappa$ is crystalline, write $\log^{++}$ for the Bloch--Kato logarithm of \cite[Section 2.1]{BSV2}, as recalled before during the proof of Proposition \ref{beta-beta}.

\begin{propo}\label{3coses}
Assume that $\Reg_{g_{\alpha}}(V_{gh}), \, \Reg_f(V_{gh}) \neq 0$. Suppose that {\em two} of the following {\em three} equalities are true modulo $L^{\times}$. Then, the third automatically holds.
\begin{enumerate}
\item[(a)] \[ \Lp^{g_{\alpha}}(f,g_{\alpha},h_{\alpha}) = \frac{\log_p(P_{\beta \beta}) \log_p(Q_{\beta \alpha}) - \log_p(Q_{\beta \beta}) \log_p(P_{\beta \alpha})}{\log_p(u_{g_{\alpha}})}. \]
\item[(b)] \[ \frac{\partial^2 \Lp^f(\hf,g_{\alpha},h_{\alpha})}{\partial x^2} \Big|_{x=x_0} = \log_p(P_{\beta \beta}) \log_p(Q_{\alpha \alpha}) - \log_p(Q_{\beta \beta}) \log_p(P_{\alpha \alpha}). \]
\item[(c)] \[ \mathcal L_{g_{\alpha}}=\log_p(u_{g_{\alpha}}). \]
\end{enumerate}
\end{propo}

\begin{proof}
The proof is based on the study of the local cohomology class $\kappa_p(f,g_{\alpha},h_{\alpha})$ introduced in the preceding sections.

Observe that (a) and (b) are equivalent to \[ \log^{-+}(\kappa_p(f,g_{\alpha},h_{\alpha})) = \frac{\log_p(P_{\beta \beta}) \log_p(Q_{\beta \alpha}) - \log_p(Q_{\beta \beta}) \log_p(P_{\beta \alpha})}{\log_p(u_{g_{\alpha}})} \pmod{L^{\times}} \] and \[ \log^{++}(\kappa_p(f,g_{\alpha},h_{\alpha})) = \log_p(P_{\beta \beta}) \log_p(Q_{\alpha \alpha}) - \log_p(Q_{\beta \beta}) \log_p(P_{\alpha \alpha}) \pmod{L^{\times}}, \] respectively, by virtue of the explicit reciprocity laws of \cite{BSV} (both in the usual version and {\it improved} version based on the techniques of Venerucci).

Let us define the local class
\begin{equation}
\kappa_0 = \frac{1}{\Xi_{g_{\alpha}} \cdot \Omega_{h_{\alpha}}} \cdot \frac{1}{\log_p(u_{g_{\alpha}})} \cdot (\log_p(P_{\beta \beta}) \cdot Q - \log_p(Q_{\beta \beta}) \cdot P) ,
\end{equation}
where we have implicitly identified a point over the elliptic curve with its image under the Kummer map; take then $\tilde \kappa = \kappa-\kappa_0$. The element $\tilde \kappa$ clearly belongs to the kernel of the Bloch--Kato logarithm $\log^{-+}$, that we have defined by \[ \log^{-+}: H^1(\mathbb Q_p, V_{fgh}) \xrightarrow{\pr^{-+}} H^1(\mathbb Q_p, V_f \otimes V_{gh}^{\alpha \beta}) \rightarrow \mathbb C_p, \] the last map being the composition of the Perrin-Riou map and the pairing with the differentials $\omega_f \otimes \eta_{g_{\alpha}} \otimes \omega_{h_{\alpha}}$. Then, $\tilde \kappa = \lambda(\log_p(P_{\beta \alpha}) \cdot Q-\log_p(Q_{\beta \alpha}) \cdot P)$. But observe that by \cite[Corollary 7.2]{BSV} we know that the cohomology class $\kappa$ lies in the balanced part for the filtration attached to $H^1(\mathbb Q_p, V_{fgh})$ and hence $\tilde \kappa$ lies in the kernel of the map $\log^{--}$  \[ \log^{--}: H^1(\mathbb Q_p, V_{fgh}) \xrightarrow{\pr^{--}} H^1(\mathbb Q_p, V_f \otimes V_{gh}^{\alpha \alpha}) \rightarrow \mathbb C_p. \]
Hence, the non-vanishing of the regulator $\Reg_{g_{\alpha}}(V_{gh})$, implies that $\tilde \kappa = 0$ and therefore $\kappa=\kappa_0$.

From the same argument and under the assumption that $\Reg_f(V_{gh})$, the second equation yields
\begin{equation}
\kappa = \frac{1}{\Omega_{g_{\alpha}} \cdot \Omega_{h_{\alpha}}} \cdot (\log_p(P_{\beta \beta}) \cdot Q - \log_p(Q_{\beta \beta}) \cdot P) \pmod{L^{\times}},
\end{equation}
where again we have identified the points with their image under the Kummer map.

Now the statement is clear. For instance, if both (a) and (b) are true, comparing the previous expressions, we get \[ \log_p(u_{g_{\alpha}}) = \frac{\Xi_{g_{\alpha}}}{\Omega_{g_{\alpha}}} \pmod{L^{\times}}. \] The proof of the other implications is equally straightforward.
\end{proof}

It would be interesting to prove an analogue result in a more general situation, beyond the case of split multiplicative reduction. The discussion around cohomology classes is still valid, but the point is that one needs a replacement for the results expressing the second derivative of $\Lp^f$ in terms of the Bloch--Kato logarithm of the cohomology class. While we can assure that the special value $\Lp^f$ is zero, it is not clear how to proceed with its derivatives.

\begin{question}
Is there a reciprocity law relating the second derivative of $\Lp^f$ (or some {\it variation} of it) with the logarithm $\log^{--}$ of the cohomology class $\kappa(f,g,h)$ in a {\it generic} situation (non exceptional zeros)?
\end{question}

\subsection{Case (a)}

We assume first that $\alpha_g \alpha_h = 1$. The results we have proved until now showing a deep interaction between the value of the derivatives of $\Lp^f(\hf,\hg,\hh)$ and the value of $\Lp^{g_{\alpha}}(\hf,\hg,\hh)$ may be applied to study new instances of the elliptic Stark conjecture.

Let us analyze some particular cases describing the exact shape of the generalized cohomology classes. For example, according to the results of \cite{BSV}, when $g$ is a theta series attached to a quadratic field where the prime $p$ is inert and $V_{gh}=V_{\psi_1} \oplus V_{\psi_2}$ with $\psi_1$ being a genus character, we have
\[ \frac{d^2 \Lp^f(\hf,g_{\alpha},h_{\alpha})}{dx^2} \Big|_{x=x_0} = \log^{++}(\kappa_p(f,g_{\alpha},h_{\alpha})) = \log(P_{\psi_1}^+) \cdot \log(P_{\psi_2}^+) \pmod{L^{\times}}, \] where $P_{\psi_i}^+ = P_{\psi_i} + \sigma_p P_{\psi_i}$, being $\sigma_p \in \Gal(H/\mathbb Q)$ a Frobenius element at $p$.

\begin{remark}
This situation occurs in general when at least one of $\psi_1$ or $\psi_2$ is a genus character. See for example the discussion after \cite[Lemma 3.10]{DLR1} where the authors explain how the regulator of the elliptic Stark conjecture admits a particularly simple expression in this case.
\end{remark}

However, from the results we already know around the elliptic Stark conjecture, one obtains that
\begin{equation}
\Lp^{g_{\alpha}}(f,g_{\alpha},h_{\alpha}) = \frac{\log(P_{\psi_1}^+) \cdot \log(P_{\psi_2}^-)}{\mathcal L_{g_{\alpha}}} \pmod{L^{\times}},
\end{equation}
where with the previous notations, $P_{\psi_i}^- = P_{\psi_i} - \sigma_p P_{\psi_i}$. This is quite significant, since it establishes the elliptic Stark conjecture only up to a conjecture about periods of weight one modular forms.

\begin{coro}\label{perintro}
Let $g$ and $h$ be theta series attached to a quadratic field (either real or imaginary) where the prime $p$ remains inert, with $V_{gh} = V_{\psi_1} \oplus V_{\psi_2}$ and $\psi_1$ or $\psi_2$ being a genus character. Then, the elliptic Stark conjecture of \cite{DLR1} is equivalent to the conjecture about periods of \cite{DR2.5}.
\end{coro}
\begin{proof}
This follows from the fact that part (b) of Proposition \ref{3coses} holds in this setting.
\end{proof}

Moreover, we expect conjectural expressions for the generalized Kato classes. In particular, the previous result suggests the following conjecture.

\begin{propo}
In the setting of Proposition \ref{3coses}, if the formulas which appear in that statement are satisfied, then the equality \[ \kappa(f,g_{\alpha},h_{1/\alpha}) = \frac{1}{\Omega_{g_{\alpha}} \cdot \Omega_{h_{\alpha}}} \cdot( \log_p(P_{\beta \beta}) \cdot Q- \log_p(Q_{\beta \beta}) \cdot P), \]
holds in $H_{\fin}^1(\mathbb Q, V_{fgh})$ up to multiplication by $L^{\times}$.
\end{propo}
\begin{proof}
This follows verbatim the proof of Proposition \ref{3coses}, using the third statement to simplify the different period relations.
\end{proof}

The same result holds for $\kappa(f,g_{\beta},h_{1/\alpha})$.

\subsection{Case (b)}

In the case where $\alpha_g \beta_h = 1$, the explicit reciprocity law gives a connection between $\Lp^{g_{\alpha}}$ and the Bloch--Kato logarithm of $\kappa(f,g_{\alpha},h_{1/\beta})$, but unfortunately both the latter class and one of the Euler factors involved in the equality vanish. Therefore, that result is meaningless in this setting.

In previous sections we saw how to overcome that difficulty, proving a {\it derived reciprocity law} after having observed that certain Euler factors are analytic along the line $k+m=\ell+2$. There are two {\em natural} directions for considering the derivative over that plane (although of course it makes sense to take any combination of them): the line $(2,\ell,\ell)$ and the line $(k+1,k,1)$; the former is not quite interesting since both the class $\kappa(f,g_{\alpha},g_{1/\beta}^*)$ and the Euler factor in the denominator of the Perrin-Riou map vanish identically. Hence, we may take derivative along $(k+1,k,1)$ and we get an equality of the form \[ \mathcal L \cdot \Lp^{g_{\alpha}}(f,g_{\alpha},h_{1/\beta}) = \log^{-+}(\kappa_p'(f,g_{\alpha},h_{1/\beta})) \pmod{L^{\times}}, \] where $\mathcal L$ is the $\mathcal L$-invariant which already appeared in previous sections. Hence, if the elliptic Stark conjecture for $\Lp^{g_{\alpha}}$ were true, the class $\kappa_p'(f,g_{\alpha},h_{\alpha})$ could be expressed as a combination of points, normalized by appropriate $\mathcal L$-invariants. In particular, this would yield an equality of the form
\begin{equation}
\kappa'(f,g_{\alpha},h_{1/\beta}) = \frac{\mathcal L}{\Xi_{g_{\alpha}} \cdot \Omega_{h_{1/\beta}}} \cdot \frac{\log_p(P_{\beta \beta}) \cdot Q - \log_p(Q_{\beta \beta}) \cdot P}{\log_p(u_{g_{\alpha}})} \pmod{L^{\times}}.
\end{equation}
One may obtain a symmetric expression for $\kappa'(f,g_{\beta},h_{1/\alpha})$. Recall that this is the analogue of \cite[Theorem B]{RiRo1}.

\begin{conj}\label{final}
The equality \[ \kappa'(f,g_{\alpha},h_{1/\beta}) = \frac{\mathcal L}{\Xi_{g_{\alpha}} \cdot \Omega_{h_{1/\beta}}} \cdot \frac{\log_p(P_{\beta \beta}) \cdot Q - \log_p(Q_{\beta \beta}) \cdot P}{\log_p(u_{g_{\alpha}})} \pmod{L^{\times}} \] holds in $H_{\fin}^1(\mathbb Q, V_{fgh})$.
\end{conj}
As it was pointed out before, in the self dual case the product $\Xi_{g_{\alpha}} \Omega_{h_{1/\beta}}$ is an element of $L^{\times}$.

We finish our work with the following result.
\begin{propo}
Assume that Conjecture \ref{final} is true. Then, the special value $\Lp^{g_{\alpha}}$ satisfies \[ \Lp^{g_{\alpha}}(f,g_{\alpha},h_{\alpha}) = \frac{\log_p(P_{\beta \beta}) \log_p(Q_{\beta \alpha}) - \log_p(Q_{\beta \beta}) \log_p(P_{\beta \alpha})}{\log_p(u_{g_{\alpha}})} \pmod{L^{\times}}. \]
\end{propo}
\begin{proof}
This follows by applying the Bloch--Kato logarithm $\log^{-+}$ to the cohomology class $\kappa'(f,g_{\alpha},h_{1/\beta})$, and using the derived reciprocity law of Theorem \ref{reclaw}.
\end{proof}

The converse can also be established with some extra assumptions, including the conjecture about periods of \cite{DR2.5}.

\end{document}